\documentclass[12pt,letterpaper]{article}

\usepackage[letterpaper,margin=1in,nomarginpar]{geometry}
\usepackage{comment}
\usepackage[title]{appendix}

\usepackage{amssymb}
\usepackage{amsthm}
\usepackage{amsmath}
\usepackage{mathtools}

\usepackage{nicefrac}

\usepackage{algorithmic}
\usepackage{algorithm}
\usepackage[shortlabels]{enumitem}
\setlist{leftmargin=*} 
\usepackage{dsfont}

\usepackage{url}

\usepackage{graphicx,wrapfig}

\usepackage{xcolor} 

\def\fprod#1{\left\langle#1\right\rangle}

\def\rev#1{{#1}}

\definecolor{mygreen}{RGB}{14,163,78}

\def\eyh#1{\textcolor{black}{#1}}

\def\aj#1{\textcolor{black}{#1}}

\DeclareMathOperator{\diag}{diag}
\DeclareMathOperator*{\argmin}{\arg\!\min}
\DeclareMathOperator*{\argmax}{\arg\!\max}
\newcommand{\dist}{\mathop{\bf dist{}}}
\def\grad{\nabla}
\def\bI{\mathbf{I}}
\def\bD{\mathbf{D}}
\def\bM{\mathbf{M}}
\def\bR{\mathbf{R}}

\def\bL{\mathbf{L}}

\def\cC{\mathcal{C}}

\def\cO{\mathcal{O}}
\def\cP{\mathcal{P}}

\def\cU{\mathcal{U}}

\def\cX{\mathcal{X}}
\def\cY{\mathcal{Y}}
\def\cZ{\mathcal{Z}}
\def\norm#1{\left\|#1\right\|}
\newcommand{\reals}{\mathbb{R}}

\def\rev#1{\textcolor{black}{#1}}
\def\revv#1{\textcolor{black}{#1}}
\newcommand{\cm}[1]{\textcolor{black}{#1}}
\def\revf#1{#1}

\newsavebox{\theorembox}
\newsavebox{\lemmabox}
\newsavebox{\defnbox}
\newsavebox{\corollarybox}
\newsavebox{\propositionbox}
\newsavebox{\remarkbox}
\newsavebox{\assbox}
\savebox{\theorembox}{\noindent\bf Theorem}
\savebox{\lemmabox}{\noindent\bf Lemma}
\savebox{\defnbox}{\noindent\bf Definition}
\savebox{\corollarybox}{\noindent\bf Corollary}
\savebox{\propositionbox}{\noindent\bf Proposition}
\savebox{\remarkbox}{\noindent\bf Remark}
\savebox{\assbox}{\noindent\bf Assumption}

\newtheorem{assumption}{Assumption}
\newtheorem{definition}{Definition}
\newtheorem{theorem}{Theorem}
\newtheorem{lemma}[theorem]{Lemma}
\newtheorem{proposition}[theorem]{Proposition}
\newtheorem{corollary}[theorem]{Corollary}
\theoremstyle{remark}
\newtheorem{remark}{Remark}

\numberwithin{assumption}{section}
\numberwithin{definition}{section}
\numberwithin{theorem}{section}
\numberwithin{remark}{section}

\usepackage{hyperref}
\hypersetup{colorlinks=true,citecolor=blue,linktocpage=true,breaklinks=true}

\allowdisplaybreaks
\title{\Large Linear Convergence of a Unified Primal--Dual Algorithm for Convex--Concave Saddle Point Problems with Quadratic Growth}

\author{
  Cody Melcher\thanks{School of Mathematical Sciences, University of Arizona, Tucson, AZ, USA. \texttt{cmelcher@arizona.edu}} 
  \quad
  Afrooz Jalilzadeh\thanks{Department of Systems and Industrial Engineering, University of Arizona, Tucson, AZ, USA. \texttt{\{afrooz, erfany\}@arizona.edu}}
  \quad
  Erfan Yazdandoost Hamedani$^{\dagger}$
}
\date{}

\begin{document}
\maketitle
\begin{abstract} 
In this paper, we study saddle point (SP) problems, focusing on convex-concave optimization involving functions that satisfy either two-sided quadratic functional growth (QFG) or two-sided quadratic gradient growth (QGG)--novel conditions tailored specifically for SP problems as extensions of quadratic growth conditions in minimization. These conditions relax the traditional requirement of strong convexity-strong concavity, thereby encompassing a broader class of problems. We propose a generalized accelerated primal-dual (GAPD) algorithm to solve SP problems with non-bilinear objective functions, unifying and extending existing methods. We prove that our method achieves a linear convergence rate under these relaxed conditions. Additionally, we provide examples of structured SP problems that satisfy either two-sided QFG or QGG, demonstrating the practical applicability and relevance of our approach.
\end{abstract}

\section{Introduction}
Let $(\cX,\|\cdot\|_\cX)$ and $(\cY,\|\cdot\|_{\cY})$  be finite-dimensional real normed vector spaces with normed dual spaces denoted by $(\cX^*,\|\cdot\|_{\cX^*})$ and $(\cY^*,\|\cdot\|_{\cY^*})$, respectively.
In this paper, we study the following saddle point (SP) problem:
\begin{align}\label{eq:main-sp}
    \min_{x \in X} \max_{y \in Y} f(x, y),
\end{align}
where $f: \cX \times \cY \to \mathbb{R}$ is convex in $x$ for any $y\in Y$ and concave in $y$ for any $x\in X$; moreover $X \subseteq \cX$ and $Y \subseteq \cY$ are closed, and convex sets. \rev{We define the Cartesian product space \rev{$\mathcal{Z}\triangleq  \mathcal{X}\times\mathcal{Y} =\{ (x, y)\ |\ x\in \mathcal{X},\ y\in \mathcal{Y}\}$} \revv{with the product norm defined as $\|z\|_{\mathcal{Z}}^2 =\|(x,y)\|_{\mathcal{Z}}^2=\|x\|_{\mathcal{X}}^2+\|y\|_{\mathcal{Y}}^2$ so that $\mathcal{Z}$ is also a finite--dimensional normed vector space.} } \revv{While this paper focuses on the class of convex--concave saddle-point problems, which is important in optimization theory, saddle-point problems more broadly arise in a wide variety of settings,} including game theory \cite{bacsar1998dynamic}, distributionally robust learning \cite{yu2022fast}, and generative adversarial networks \cite{goodfellow2014generativeadversarialnetworks,pmlr-v70-arjovsky17a}. Numerous algorithms and approaches, such as extragradient (EG) \cite{korpelevich1976extragradient}, gradient descent-ascent (GDA) \cite{popov1980modification}, optimistic gradient descent-ascent (OGDA) \cite{daskalakis2017training}, and the accelerated primal-dual method (APD) introduced in \cite{hamedani2021primal}, have been developed to tackle these challenges. Primal-dual methods, in particular, stand out for their stability, adaptability, and suitability for large-scale optimization problems. This paper will center on an accelerated primal-dual approach, which achieves faster convergence by incorporating a momentum term to blend insights from prior and current iterates.

For saddle point problems, linear convergence is well-known to be guaranteed when the objective function is strongly convex-strongly concave, \revv{(i.e., the objective is $\mu$--strongly convex if 
$f(y)\ge f(x)+\langle\nabla f(x),y-x\rangle+\frac{\mu}{2}\|y-x\|^2$ 
for some $\mu>0$, and $\mu$--strongly concave if $-f$ satisfies the same property)} along with either smoothness or efficient proximal mappings \cite{jiang2022generalized}. 
{Recent work by \cite{necoara2019linear} introduced several relaxations of strong convexity in the context of minimization problems, notably the \textit{Quadratic Functional Growth} (QFG) and \textit{Quadratic Gradient Growth} (QGG) conditions, both of which guarantee a linear convergence rate. These conditions ensure that the functional values or gradients of the objective grow at least quadratically as they move away from the optimal set. However, extending these relaxations to saddle point problems is a non-trivial challenge and remains largely unexplored.}

{In this paper, we extend the definitions of QFG and QGG to the saddle point problem framework, introducing the notions of two-sided QFG and two-sided QGG (see Definitions \ref{qgg} and \ref{qfg}). Our analysis demonstrates that the proposed accelerated primal-dual algorithm achieves a linear convergence rate when the objective satisfies either two-sided QFG or QGG, paralleling the results observed under strong convexity. Furthermore, we enhance our framework by incorporating Bregman divergence instead of Euclidean distance. This refinement, combined with the two-sided QFG or QGG conditions, allows us to achieve linear convergence under a broader and more flexible set of assumptions.}

\subsection{Related Work}

\noindent \aj{{\bf Minimization Problems.}
Function classes satisfying QFG and QGG have shown to be valuable in both deterministic and stochastic optimization settings.} For QGG in a deterministic setting, \cite{meng2020aug} consider a smooth, convex optimization problem with convex nonlinear inequality constraints. Assuming the objective satisfies QGG and the problem has a unique optimum, the authors achieve a linear convergence rate 
by constructing an augmented primal-dual algorithm. \cite{yang2024variance} consider a similar smooth convex optimization problem, but from a stochastic approach, with the added assumption of the objective having a finite sum structure. By utilizing a variance-reduced moving balls approximation method and assuming each $f_i$ satisfies QGG, a linear convergence rate is achieved. There are also several studies incorporating the QFG condition in stochastic regime achieving a sublinear convergence, e.g., \cite{grimmer2019convergence} and \cite{khaled2020better}. 

\noindent{\bf Saddle Point Problems.}
Saddle point problems have been a central focus of optimization research for decades, driven both by their unique structure, which sparks theoretical interest, and their natural occurrence in applications. 

The earliest methods for solving constrained saddle point problems date back to the Arrow-Hurwicz method \cite{arrow1958studies}, also known as gradient descent ascent (GDA). While this method laid the groundwork, its convergence is suboptimal for smooth and strongly-convex-strongly-concave problems, achieving an iteration complexity of \(\cO(\kappa^2 \log(1/\epsilon))\) \cite{nesterov2006solving,faechinei2003finite}. To address these limitations,  \cite{korpelevich1976extragradient} and  \cite{popov1980modification} proposed modifications that introduced extrapolation and optimism, respectively.
Korplevich's extragradient (EG) method \cite{korpelevich1976extragradient} added an extrapolation step to the Arrow-Hurwicz framework, enabling linear convergence for smooth and strongly-convex-strongly-concave functions. This method achieves an iteration complexity of \(\cO(\kappa \log(1/\epsilon))\) \cite{tseng1995linear,azizian2020tight,mokhtari2020a}. Subsequent generalizations include the mirror-prox method \cite{nemirovski2004prox}, which extends to Bregman distances, and variations accommodating unbounded feasible sets and composite objectives \cite{monteiro2011complexity,monteiro2012iteration}.
In parallel, the dual extrapolation method \cite{nesterov2007dual} addressed variational inequalities (VIs) and SP problems by performing extrapolation in the dual space. This approach also achieves linear convergence for smooth, strongly monotone VIs \cite{nesterov2011solving}.
Tseng's splitting method \cite{tseng2000modified} further expanded the field by tackling monotone inclusion problems, demonstrating linear convergence for strongly monotone cases under the Lipschitz continuity assumption of the operator. Extensions utilizing the hybrid proximal extragradient (HPE) framework \cite{monteiro2011complexity} have shown similar efficacy for smooth monotone inclusion problems and convex-concave SP problems.
More recently, Du and Hu \cite{du2018linear} analyzed a bilinear SP problem under C-SC setting with a full column rank coupling matrix, and demonstrated that standard GDA can still achieve a linear convergence in this setting.

Popov's method \cite{popov1980modification}, also known as optimistic gradient descent ascent (OGDA) \cite{daskalakis2017training}, introduced a single-sequence update mechanism, requiring only one gradient computation per iteration. For smooth and strongly-convex-strongly-concave saddle point problems, OGDA achieves an iteration complexity of \(\cO(\kappa \log(1/\epsilon))\) \cite{mokhtari2020b}. By approximating the proximal point method, OGDA also achieves an ergodic complexity of \(\cO(1/\epsilon)\) for unconstrained convex-concave saddle point problems \cite{mokhtari2020b}.
Recent advancements have extended OGDA to Bregman distances, maintaining its linear convergence guarantees for convex-concave and strongly-convex-strongly-concave settings \cite{kotsalis2022simple,jiang2022generalized}. 

Another line of research focuses on a widely-used family of algorithms known as primal-dual hybrid gradient methods, which are tailored to the structure of SP problems by alternately updating the variables \(x\) and \(y\) with the incorporation of momentum terms \cite{chambolle2011first,condat2013primal,chen2014optimal,chambolle2016ergodic,hamedani2021primal}. Initiated by \cite{chambolle2011first}, most of these methods are designed for SP problems with bilinear objective functions, achieving linear convergence rates under the SC-SC setting. A notable exception is the accelerated primal-dual (APD) method \cite{hamedani2021primal}, which extends to general non-bilinear SP problems. The APD method achieves complexities of \(\mathcal{O}(1/\epsilon)\) and \(\mathcal{O}(1/\epsilon^2)\) under C-C and SC-C\footnote{$f(x,\cdot)$ is assumed to be linear.} settings, respectively. 
{Finally, \cite{zamani2024convergence} prove a linear convergence rate for standard gradient descent–ascent under a quadratic gradient growth condition for unconstrained saddle point problems in the Euclidean setting.}

\noindent\textbf{Variational Inequalities.}
\rev{Quadratic growth conditions are well established in the VIs literature, where they are closely linked to error bound analysis and (strong) metric subregularity of monotone operators (see \cite{drusvyatskiy2013tilt}, \cite{drusvyatskiy2015quadratic}, \cite{artacho2008characterization}).
This connection has been used to prove linear convergence of proximal methods in unconstrained minimization problems \cite{drusvyatskiy2018error}, and has also been employed in the context of primal-dual methods for affine-coupled composite minimization and in the analysis of the alternating direction method of multipliers (ADMM) for linearly constrained problems (Ch. 8, \cite{Jin2025-go}). Both of these settings, however, represent special cases of the broader smooth convex-concave saddle-point framework we focus on. While any convex–concave saddle-point problem can be reformulated as a VI, existing analyses in that framework typically operate through monotone operator theory and strong monotonicity assumptions. In contrast, our approach works directly with the saddle-point objective and relies on a new two-sided QGG and QFG definitions, which are strictly weaker than strong monotonicity, but still ensure linear convergence.}

\paragraph{Contribution}
We propose a Generalized Accelerated Primal-Dual (GAPD) algorithm to address constrained SP problems with non-bilinear objective functions, unifying and extending existing methods such as APD \cite{hamedani2021primal} and OGDA \cite{mokhtari2020b}. Assuming a convex-concave and smooth objective function, our approach fills a critical gap in the literature by identifying conditions for achieving a linear convergence rate {for Problem \ref{eq:main-sp}} without requiring strong convexity-strong concavity. Specifically, we establish that a linear convergence to the SP solution set can be achieved under either of two weaker conditions: two-sided Quadratic Functional Growth (QFG) or two-sided Quadratic Gradient Growth (QGG). Moreover, the proposed method accommodates Bregman distances, significantly enhancing its flexibility and applicability. To further validate our contributions, we provide examples of structured SP problem classes that satisfy these relaxed growth conditions.

\section{Assumptions and Definitions}
This section outlines our notations, the key assumptions, and the definitions used in the analysis. First, we fix the notations used in the paper.
\paragraph{Notation} Throughout the paper, \rev{we reserve uppercase calligraphic letters $\mathcal{U}, \mathcal{X}, \mathcal{Y}, \mathcal{Z}$ to denote finite-dimensional real normed vector spaces, while the corresponding uppercase Roman letters $U, X, Y, Z$ represent subsets of these spaces.
For any finite-dimensional real normed space $\cU$, we denote its dual space by $\cU^*$. We denote the norms on $\cU$ and $\cU^*$ by $\|\cdot\|_\cU$ and $\|\cdot\|_{\cU^*}$, respectively. In addition, we denote the duality pairing between $\cU^*$ and $\cU$ by $\langle \cdot, \cdot \rangle : \cU^* \times \cU \to \mathbb{R}$. 
Further, $\norm{\cdot}_2$ denotes either the Euclidean norm (for vectors) or the spectral norm (for matrices).} Given a convex function $g:\reals^n\rightarrow\reals\cup\{+\infty\}$, its convex conjugate is defined as $g^*(w)\triangleq\sup_{\theta\in\reals^n}\fprod{w,\theta}-g(\theta)$. 
$\mathbb{R}^n_{++}$ ($\mathbb{R}^n_{+}$) and $\mathbb{S}^n_{++}$ ($\mathbb{S}^n_{+}$) denote the positive (nonnegative) orthant and the set of positive (semi)definite matrices, respectively, moreover, $\bI_n$ is the $n\times n$ identity matrix. For a block diagonal matrix $A=\diag(\{a_i\bI_{n_i}\})\in\mathbb{S}^{n}_{++}$ where $\{a_i\}_{i=1}^m\subset\reals_{++}$ and $n=\sum_{i=1}^m n_i$, $Q$-inner product and $Q$-norm are defined as $\fprod{z,\bar{z}}_Q=\sum_{i=1}^m a_i\fprod{z_i,\bar{z_i}}$ and $\norm{z}^2_Q\triangleq \sum_{i=1}^m a_i\norm{z_i}^2$.
Finally, for $\theta\in\reals^n$, we adopt $[\theta]_+\in\reals^n_{+}$ to denote $\max\{\theta, \textbf{0}\}$ where max is computed componentwise. 
$g(n)=\Theta(f(n))$ means $\exists n_0\in\mathbb{Z}_+$ and $c_1,c_2>0$ such that $c_1 f(n)\leq g(n)\leq c_2 f(n)$ for $n\geq n_0.$

Now, we present the assumptions regarding the objective function.

\rev{
\begin{assumption}[Convex-Concave Structure and Smoothness] \label{assumption:convex-concave-smoothness}
The function $f:\cX\times\cY\to\mathbb{R}$ is continuously differentiable on an open set containing $X\times Y$ and satisfies the following,
\begin{enumerate}[(i)]
    \item 
    For any $y\in\cY$, $f(\cdot,y)$ is convex, and there exist $L_{xx} \ge 0$ and $L_{xy} \ge 0$ such that for all $x,\hat{x}\in\cX$ and $y,\hat{y}\in\cY$,
    \begin{equation} \label{eq:Lip-x}
        \|\nabla_x f(x, y) - \nabla_x f(\hat{x}, \hat{y})\|_{\cX^*} \;\le\; L_{xx} \|x - \hat{x}\|_\cX + L_{xy} \|y - \hat{y}\|_\cY.
    \end{equation}
    \item 
    For any $x\in\cX$, $f(x,\cdot)$ is concave, and there exist $L_{yy} \ge 0$ and $L_{yx} > 0$ such that for all $x,\hat{x}\in\cX$ and $y,\hat{y}\in\cY$,
    \begin{equation} \label{eq:Lip-y}
        \|\nabla_y f(x, y) - \nabla_y f(\hat{x}, \hat{y})\|_{\cY^*} \;\le\; L_{yy} \|y - \hat{y}\|_\cY + L_{yx} \|x - \hat{x}\|_\cX.
    \end{equation}
\end{enumerate}
\end{assumption}
Note that the subscripts on the Lipschitz constants only specify with which variable the derivative of the gradient was taken; they do not indicate any sort of functional dependence. For example, $L_{xx}$ is the Lipschitz constant associated with the partial derivative with respect to $x$ of the gradient's $x$ component. 
Moreover, \eqref{eq:Lip-x} and convexity imply that for any $y\in \mathcal Y$, and $x,\hat{x}\in\mathcal X$,
\begin{equation}\label{eq:lip-x-extend}
0\leq f(x,y)-f(\hat{x},y)-\fprod{\nabla_x f(\hat{x},y),x-\hat{x}}\leq \frac{L_{xx}}{2}\|x-\hat{x}\|_{\mathcal X}^2,
\end{equation}
and similarly \eqref{eq:Lip-y} and concavity imply that 
\begin{equation}\label{eq:lip-y-extend}
0\geq f(x,y)-f(x,\hat{y})-\fprod{\nabla_y f(x,\hat{y}),y-\hat{y}}\geq -\frac{L_{yy}}{2}\|y-\hat{y}\|_{\mathcal Y}^2.
\end{equation}
}

\aj{One of the contributions of this work is the use of Bregman distance, which generalizes the Euclidean distance, offering greater flexibility and broader applicability in the analysis. Next, we define the concepts of Bregman distance and Bregman projection.}
\begin{definition}[Bregman Distance]\label{def-bregman}
Let $\varphi: \cU \to \mathbb{R}$ be a differentiable, 1-strongly convex function on the domain $\cU$. The \textit{Bregman distance} $\bD_\cU: \cU \times \cU \to \mathbb{R}_+$ associated with $\varphi$ is defined as:
\[
\bD_\cU(z, \hat{z}) \triangleq \varphi(z) - \varphi(\hat{z}) - \langle \nabla \varphi(\hat{z}), z - \hat{z} \rangle,
\]
where $z, \hat{z} \in \cU$. The Bregman distance is non-negative and equals zero if and only if $z = \hat{z}$. Moreover, if $\cU= \prod_{i=1}^\ell \cU_i$, where $\cU_i$ is a finite-dimensional ($n_i$) \rev{normed vector space and therefore $\mathcal{U}$ is also a finite-dimensional normed vector space,} and $A=\diag(\{a_i\bI_{n_i}\})$ for some $\{a_i\}_{i=1}^\ell\subset \reals_{++}$, then we define $\bD_\cU^A(u,\hat{u})\triangleq \sum_{i=1}^\ell a_i\bD_{\cU_i}(u_i,\hat{u}_i)$. 
\end{definition}

\begin{remark}\label{bregineq} By definition,  $\bD_\cU(z, \hat{z}) \geq \frac{1}{2} \| z - \hat{z} \|^2_\cU$ for $z \in \cU$ and $\hat{z} \in \operatorname{dom} f$. \revv{Note that the $1$-strong-convexity property is stated only for the Bregman distance generator function $\varphi$, and the scaling by diagonal matrix $A$ with positive elements is treated as a re-weighting of the resulting distance, not as a strong-convexity statement about the scaled generator}. Further, if we choose $\varphi$ to be $\frac{1}{2}\|\cdot\|^2$, then the Bregman distance becomes the Euclidean distance.\end{remark}

\begin{definition}[Bregman Projection]\label{def-bregman-projection}
We define the Bregman projection of a point $\hat{u} \in \cU$ onto a convex and closed set $C \subseteq \cU$ as:
\[
\cP_{\cC}(\hat{u}) \triangleq \argmin_{u \in C} \bD_\cU(u, \hat{u}).
\]
Moreover, we define $\dist(\revf{\hat u},C)\triangleq \bD_\cU(\cP_{C}(\hat{u}), \hat{u})$. 
\end{definition}

We present the definitions of Quadratic Gradient Growth (QGG) and Quadratic Functional Growth (QFG), following \cite{necoara2019linear}, with an extension that incorporates Bregman distance functions.

\begin{definition}[QGG]\label{nec-qgg}
Consider a minimization problem $\min_{x\in X}g(x)$ where $g:\cX\to\reals$ is a
continuously differentiable function. Let $X^*$ denote the set of minimizers of $g$. Then we say that \( g \) satisfies \textit{quadratic gradient growth} relative to the set \( X^* \) if there exists a constant \( \kappa_g > 0 \) such that for any \( x \in X \) and \( \bar{x} = \cP_{X^*}(x)\)  we have:
\[
\langle \nabla g(x) - \nabla g(\bar{x}), x - \bar{x} \rangle \geq 2\kappa_g \bD_\cX(\bar{x},x)
\]
\end{definition}

\begin{definition}[QFG]\label{nec-qfg}
Consider a minimization problem $\min_{x\in X}g(x)$ where $g:\cX\to\reals$ is a
continuously differentiable function. Let $X^*$ denote the set of minimizers of $g$. Then we say that \( g \) satisfies \textit{quadratic functional growth} relative to the set \( X^* \) if there exists a constant \( \kappa_g > 0 \) such that for any \( x \in X \) and \( \bar{x} = \cP_{X^*}(x)\)  we have:
\[ 
g(x)-g^* \geq \kappa_g \bD_\cX(\bar{x},x).
\]
\end{definition}

\begin{definition}[Saddle Point Solution]\label{optset}
For a convex-concave function $f$ in \eqref{eq:main-sp}, the SP solution set \( Z^* \subseteq \mathcal{Z} := \mathcal{X} \times \mathcal{Y} \) is defined as
\[
Z^* := \left\{ (x^*, y^*) \in \revf{X \times Y} : f(x^*, y) \leq f(x^*, y^*) \leq f(x, y^*), \;\; \forall (x, y) \in X \times Y \right\}.
\]
\end{definition}
Based on Sion's theorem \cite{sion1958general} and Proposition 3.4.1 in \cite{bertsekas2009convex}, when $f(\cdot,\cdot)$ is convex-concave and $Y$ is compact, one can represent the saddle point solution set as $Z^*=X^*\times Y^*$ where $X^*\triangleq \argmin_{x\in X}\max_{y\in Y}f(x,y)$ and $Y^*\triangleq \argmax_{y\in Y}\min_{x\in X}f(x,y)$. \revv{Moreover, for any $A=\diag(a_x I_n,a_y I_m)$, where $a_x,a_y>0$ and for any $z_1=(x_1,y_1)$ and $z_2=(x_2,y_2)$ we define $\bD_{\cZ}^A(z_1,z_2)
= a_x \bD_{\cX}(x_1,x_2)+a_y \bD_{\cY}(y_1,y_2)$.
Since both the distance and the solution set are separated across the two blocks, the projection over $Z^*$ decouples as $\argmin_{u\in Z^*} \bD_{\cZ}^A(u,z)
=
\bigl(\cP_{X^*}(x),\cP_{Y^*}(y)\bigr) = \cP_{Z^*}(z)$ 
.}
We now define \textit{two-sided} QGG and QFG conditions designed for our saddle-point setting.
\begin{definition}[Two-Sided QGG]\label{qgg}
Consider problem \eqref{eq:main-sp} and let $Z^*$ denote the set of saddle point solutions. 
Then, \( f(x, y) \) has a \textit{two-sided quadratic gradient growth} on $Z^*$ if there exist constants \((\mu_x,\mu_y)\in\reals^2_{++}\)
such that for any \( x \in X \) and \( y \in Y \), the following inequality holds:
\begin{equation}\label{eq:qgg}
    \langle F(z)-F(\bar{z}), z - \bar{z} \rangle \geq 2\bD_{\cZ}^{\bM}(\bar z,z),
\end{equation}
where $z\triangleq [x^\top \ y^\top]^\top$, $\bar{z}\triangleq \cP_{Z^*}(z)$, $F(z)\triangleq [\grad_x f(x,y)^\top \ -\!\grad_y f(x,y)^\top]^\top
$, and $\bM\triangleq \diag(\{\mu_x\bI_n,\mu_y\bI_m\})$. 
\end{definition}
\begin{definition}[Two-Sided QFG]\label{qfg}
Consider problem \eqref{eq:main-sp} and let $Z^*$ denote the set of saddle point solutions. 
Then, \( f(x, y) \) has a \textit{two-sided quadratic functional growth} on $Z^*$ if there exist constants \((\mu_x,\mu_y)\in\reals^2_{++}\)
such that for any \( x \in X \) and \( y \in Y \), the following inequality holds:
\begin{equation}\label{eq:qfg}
\revv{f(x,{y}^*) - f({x}^*,y) \geq \bD_{\cZ}^{\bM}(\bar{z},z),\quad \forall (x^*,y^*)\in Z^*,}
\end{equation}
where $\bar z=\cP_{Z^*}(z)$.
\end{definition}

\begin{remark}
  As shown in \cite{necoara2019linear}, the QFG condition in Definition \ref{nec-qfg} for minimization problems is weaker than the QGG condition outlined in Definition \ref{nec-qgg} in the Euclidean setting.
  However, this relationship does not necessarily hold for two-sided definitions with Bregman distance. Specifically, the two-sided growth conditions in Definitions \ref{qgg} and \ref{qfg} do not imply one another without additional assumptions, as we show in the following examples. This distinction underscores the nontrivial extension of quadratic growth conditions from minimization problems to saddle point problems.
\end{remark}

\rev{First, we demonstrate that two-sided QFG does not imply two-sided QGG. Then, through a specific example, we show that two-sided QGG does not imply two-sided QFG, using a carefully chosen Bregman distance function.}\\ 

\rev{
\noindent \textbf{Example 1.} Consider the following SP problem in the Euclidean space, 
  \begin{equation*}
\min_{x\in[0,1]^2} \max_{y\in[0,1]^2} \frac{1}{2}(x_1^2 + x_2^2) +(x_1+x_2)y_1 - y_1^2 +y_2. 
  \end{equation*}
One can verify that the problem has a unique saddle point at $(x_1^*,x_2^*)=(0,0),(y_1^*,y_2^*)=(0,1)$. Recall  from Definition \ref{qgg} that $(\bar{x},\bar{y})$ denotes the projection of any point $(x,y)$ onto the solution set, hence, $(\bar{x},\bar{y})=({x}^*,{y}^*)$. 
Now, we show that $f$ satisfies two-sided QFG with respect to $\frac{1}{2}\|\cdot\|_2^2$.  
\begin{align*}
f(x, y^*) - f(x^*, y) 
&= \tfrac{1}{2}(x_1^2 + x_2^2) + 1 + y_1^2 - y_2 \\
&= \tfrac{1}{2}(x_1^2 + x_2^2) + (1 + y_1^2 - y_2) \\
&\geq \tfrac{1}{2}(x_1^2 + x_2^2) + \tfrac{1}{2}(y_1^2 + (y_2 - 1)^2)
= \frac{1}{2} \|x-\bar{x}\|_\mathcal{X}^2 + \frac{1}{2} \|y-\bar{y}\|_\mathcal{Y}^2,
\end{align*}
where the inequality follows from the fact that $y_1^2\geq \frac{1}{2}y_1^2$ and $1-y_2\geq \frac{1}{2}(y_2-1)^2$ for any $y_1,y_2\in[0,1]$. 
To show this example does not satisfy two-sided QGG, consider the point $z=(x,y)=\big((0,0),(0,\frac{1}{2})\big)$ and $\bar{z}=z^*$, then we observe that for any $\mu_x,\mu_y>0$,  
\begin{align*}
    \langle F(z)-F(\bar{z}),z-\bar{z}\rangle=0 \not\geq 2\bD^\bM_\mathcal{Z}(\revf{\bar{z},z})= \revf{\mu_x\|x-\bar{x}\|^2_2+\mu_y\|y-\bar{y}\|^2_2 =\frac{\mu_y}{4}}.\qed 
\end{align*}}%

\rev{In the following example, we show that the two-sided QGG property defined in Definition \ref{qgg} does not imply the two-sided QFG property defined in Definition \ref{qfg}. This is due to the use of an arbitrary Bregman distance function in the definition: with general Bregman distance functions, QGG does not necessarily imply QFG because the geometry affects the lower bounds for gradient monotonicity and suboptimality differently. 
To illustrate this, we present a counterexample for a minimization problem using a specific Bregman distance function (the argument can be extended to saddle-point problems). This example does not contradict \cite[Theorem 4]{necoara2019linear}, as that result uses the Euclidean distance in its proof. Later, in Theorem \ref{thm:2QGG to 2QFG}, we show that if the gradients of the distance-generating functions are Lipschitz continuous, then two-sided QGG indeed implies two-sided QFG.}\\

\noindent\textbf{Example 2.} 
Let $h:[0,\infty) \to \reals$ be defined as $h(x) = \frac{1}{4}x^2+\exp(x^2)-\frac{1}{2}\mathrm{Ei}(x^2)+\log(x)$ if $x>0$, and $h(0) = 1-\gamma/2$, where $\gamma$ is Euler's constant and $\mathrm{Ei}(x) = \int_{-\infty}^x \frac{\exp(t)}{t} \, dt$ is the exponential integral. Since $\mathrm{Ei}(x^2)=\gamma+2\log(x)+\mathcal{O}(x^2)$ as $x\downarrow 0$, the function $h$ is continuous at $0$. Let the Bregman distance-generating function be $\psi(x) = \frac{1}{2}x^2 + \exp(x^2)$, which is $1$-strongly convex on $[0,\infty)$ but whose derivative is not Lipschitz continuous. 
For $x>0$ we have $h'(x) = \frac{x}{2} +2x\exp(x^2)- \frac{\exp(x^2)-1}{x}$, and $h'(0)=0$, so $h$ is continuously differentiable. Moreover, with $\bar{x}=0$, one can verify that $h'(x)\,x=\bD(\bar{x},x)$ for all $x\geq 0$, where
\[
\bD(\bar{x},x)=\psi(0)-\psi(x)-\psi'(x)(0-x)=1+\tfrac{1}{2}x^2+(2x^2-1)\exp(x^2)\geq 0.
\]
Hence $h'(x)\geq 0$, and since $\frac{d}{dx}\bD(\bar x,x)=x\psi''(x)$, for $x>0$ we have $h''(x)=\frac{x^2\psi''(x)-\bD(\bar{x},x)}{x^2}=\frac{\frac{1}{2}x^2+(4x^4+1)\exp(x^2)-1}{x^2}\geq 0$. Therefore, $h$ is convex and nondecreasing, its minimum is attained at the unique point $x^*=0$, and $\bar{x} = \mathcal{P}_{X^*}(x) = 0$ for all $x\geq 0$. 
With these definitions, $h$ satisfies QGG with $\kappa_h=\frac{1}{2}$ since for any $x \in [0,\infty)$:  
$(h'(x) - h'(x^*))(x - x^*) = h'(x)\, x = \bD(\bar{x},x).$ 
However, QFG does not hold. By L'H\^opital's rule:  
$$
\lim_{x \to \infty} \frac{h(x) - h(x^*)}{\bD(\bar{x},x)}  
= \lim_{x \to \infty} \frac{h'(x)}{x\psi''(x)}= \lim_{x \to \infty} \frac{1+\frac{1}{2}x^2+(2x^2-1)\exp(x^2)}{x^2+(2x^2+4x^4)\exp(x^2)} = 0.
$$
Therefore, for any $\mu > 0$, there exists $x > 0$ such that  
$h(x) - h(x^*) < \mu \bD(\bar{x},x)$, which violates the QFG property. \qed

\begin{theorem}\label{thm:2QGG to 2QFG}
Suppose Assumption \ref{assumption:convex-concave-smoothness} holds and the Bregman distance generating functions $\psi_\cX$ and $\psi_\cY$ \revf{have} Lipschitz continuous gradient\revf{s} with constants $L_{\psi_\cX}$ and $L_{\psi_\cY}$, respectively. 
If $f$ satisfies the two-sided QGG in \eqref{eq:qgg} with constant $(\mu_x,\mu_y)$, then it satisfies the two-sided QFG in \eqref{eq:qfg} with constant \revf{$\Big(\frac{\mu_x^2}{L_{xx}L_{\psi_\cX}},\frac{\mu_y^2}{L_{yy}L_{\psi_\cY}}\Big)$. Moreover, if the distance generating functions are quadratic, i.e., $\psi_\cX(x)=\frac{1}{2}\fprod{Q_\cX x,x}$ and $\psi_\cY(y)=\frac{1}{2}\fprod{Q_\cY y,y}$ for some self-adjoint positive definite operators $Q_\cX$ and $Q_\cY$, then $f$ satisfies the two-sided QFG with the sharper constant $(\mu_x,\mu_y)$.}
\end{theorem}
\begin{proof}
We establish the bound for the $x$-block; the $y$-block follows from the same argument applied to the convex function $-f(x^*,\cdot)$. Let $(x^*,y^*)\in Z^*$ be an arbitrary saddle point solution pair, let $x\in X$, and define $g(u)\triangleq f(u,y^*)$ for $u\in\cX$. By Assumption \ref{assumption:convex-concave-smoothness}, $g$ is convex with an $L_{xx}$-Lipschitz continuous gradient on $\cX$. Since $Z^*=X^*\times Y^*$, every $u\in X^*$ satisfies $(u,y^*)\in Z^*$; hence, every point of $X^*$ minimizes $g$ over $X$ and $g(u)=f^*$ for all $u\in X^*$. We assume $L_{xx}>0$; if $L_{xx}=0$, then $\nabla g$ is constant, the left-hand side of \eqref{eq:qgg} at $z=(x,y^*)$ vanishes, which forces $x=\cP_{X^*}(x)$ for every $x\in X$, i.e., $X=X^*$, and the $x$-block of the QFG inequality holds trivially.

\noindent\emph{Step 1: $\nabla g$ is constant on $X^*$.} Let $a,b\in X^*$. The optimality of $a$ and $b$ for $\min_{u\in X}g(u)$ gives $\fprod{\nabla g(a),b-a}\geq 0$ and $\fprod{\nabla g(b),a-b}\geq 0$, hence $\fprod{\nabla g(a)-\nabla g(b),a-b}\leq 0$. On the other hand, since $g$ is convex with an $L_{xx}$-Lipschitz continuous gradient on $\cX$, it is co-coercive (see, e.g., \cite[Theorem 5.8]{beck2017first}), i.e.,
\begin{equation}\label{eq:cocoercive}
\fprod{\nabla g(u)-\nabla g(v),u-v}\geq \frac{1}{L_{xx}}\|\nabla g(u)-\nabla g(v)\|_{\cX^*}^2,\quad \forall u,v\in\cX.
\end{equation}
Therefore, $\nabla g(a)=\nabla g(b)$, and we denote this common value by $\nabla g^*$.

\noindent\emph{Step 2: an error bound.} For any $u\in X$, let $\bar u\triangleq \cP_{X^*}(u)$ and $d(u)\triangleq \min_{w\in X^*}\|u-w\|_\cX$. Applying \eqref{eq:qgg} at $z=(u,y^*)$, noting that $\cP_{Y^*}(y^*)=y^*$, and using Remark \ref{bregineq}, we obtain
\begin{align*}
\mu_x\|u-\bar u\|_\cX^2\leq 2\mu_x\bD_\cX(\bar u,u)&\leq \fprod{\nabla g(u)-\nabla g(\bar u),u-\bar u}\\
&\leq \|\nabla g(u)-\nabla g^*\|_{\cX^*}\|u-\bar u\|_\cX,
\end{align*}
which implies that $\|\nabla g(u)-\nabla g^*\|_{\cX^*}\geq \mu_x\|u-\bar u\|_\cX\geq \mu_x d(u)$.

\noindent\emph{Step 3: integration along a segment.} Let $\hat x\in\argmin_{w\in X^*}\|x-w\|_\cX$, which exists since $X^*$ is nonempty, closed, and convex, and let $x_t\triangleq \hat x+t(x-\hat x)\in X$ for $t\in[0,1]$. We have $d(x_t)=t\,d(x)$. Indeed, $d(x_t)\leq\|x_t-\hat x\|_\cX=t\,d(x)$, and if some $w\in X^*$ satisfied $\|x_t-w\|_\cX<t\,d(x)$, then $\|x-w\|_\cX\leq \|x-x_t\|_\cX+\|x_t-w\|_\cX<(1-t)d(x)+t\,d(x)=d(x)$, which is a contradiction. Using the fundamental theorem of calculus, we obtain
\begin{align*}
g(x)-f^*=g(x)-g(\hat x)&=\int_0^1\fprod{\nabla g(x_t),x-\hat x}\,dt\\
&=\fprod{\nabla g^*,x-\hat x}+\int_0^1\frac{1}{t}\fprod{\nabla g(x_t)-\nabla g^*,x_t-\hat x}\,dt\\
&\overset{(a)}{\geq} \int_0^1\frac{1}{t\,L_{xx}}\|\nabla g(x_t)-\nabla g^*\|_{\cX^*}^2\,dt\\
&\overset{(b)}{\geq} \int_0^1\frac{\mu_x^2\,t}{L_{xx}}\,d(x)^2\,dt=\frac{\mu_x^2}{2L_{xx}}\,d(x)^2,
\end{align*}
where in $(a)$ we used $\fprod{\nabla g^*,x-\hat x}\geq 0$ (optimality of $\hat x$) and \eqref{eq:cocoercive} with $(u,v)=(x_t,\hat x)$ together with $\nabla g(\hat x)=\nabla g^*$, and in $(b)$ we used Step 2 at $u=x_t$ and $d(x_t)=t\,d(x)$. Finally, since $\bar x=\cP_{X^*}(x)$ minimizes $\bD_\cX(\cdot,x)$ over $X^*$ and $\nabla\psi_\cX$ is $L_{\psi_\cX}$-Lipschitz continuous, $\bD_\cX(\bar x,x)\leq \bD_\cX(\hat x,x)\leq \frac{L_{\psi_\cX}}{2}\|x-\hat x\|_\cX^2=\frac{L_{\psi_\cX}}{2}d(x)^2$. Therefore,
\begin{equation*}
f(x,y^*)-f^*\geq \frac{\mu_x^2}{L_{xx}L_{\psi_\cX}}\bD_\cX(\bar x,x).
\end{equation*}
Similarly, applying the same argument to $-f(x^*,\cdot)$, we obtain $f^*-f(x^*,y)\geq \frac{\mu_y^2}{L_{yy}L_{\psi_\cY}}\bD_\cY(\bar y,y)$. Adding the two inequalities and recalling that $(x^*,y^*)\in Z^*$ is arbitrary, we conclude that $f$ satisfies two-sided QFG with constant $\big(\frac{\mu_x^2}{L_{xx}L_{\psi_\cX}},\frac{\mu_y^2}{L_{yy}L_{\psi_\cY}}\big)$.

\noindent\emph{Quadratic distance generating functions.} If $\psi_\cX(x)=\frac{1}{2}\fprod{Q_\cX x,x}$, then $\bD_\cX(w,u)=\frac{1}{2}\|w-u\|_{Q_\cX}^2$ with $\|v\|_{Q_\cX}^2\triangleq\fprod{Q_\cX v,v}$, so the Bregman projection onto $X^*$ coincides with the (unique) nearest point in the norm $\|\cdot\|_{Q_\cX}$. Then the argument used in Step 3 to show $d(x_t)=t\,d(x)$, applied with the norm $\|\cdot\|_{Q_\cX}$, shows that $\cP_{X^*}(\bar x+t(x-\bar x))=\bar x$ for all $t\in[0,1]$, where $\bar x=\cP_{X^*}(x)$, and $\bD_\cX(\bar x,\bar x+t(x-\bar x))=t^2\bD_\cX(\bar x,x)$. Hence, applying \eqref{eq:qgg} at $z=(\bar x+t(x-\bar x),y^*)$,
\begin{align*}
f(x,y^*)-f^*&=\fprod{\nabla_x f(\bar x,y^*),x-\bar x}\\
&\quad+\int_0^1\frac{1}{t}\fprod{\nabla_x f(\bar x+t(x-\bar x),y^*)-\nabla_x f(\bar x,y^*),t(x-\bar x)}\,dt\\
&\geq \int_0^1\frac{2\mu_x}{t}\,t^2\,\bD_\cX(\bar x,x)\,dt=\mu_x\bD_\cX(\bar x,x).
\end{align*}
The $y$-block is analogous, which gives the constant $(\mu_x,\mu_y)$.
\end{proof}

\section{Proposed Method}
\aj{In this section, we present the Generalized Accelerated Primal-Dual (GAPD) algorithm, a first-order method designed to solve saddle-point problems under relaxed growth conditions, such as two-sided Quadratic Functional Growth (QFG) or Quadratic Gradient Growth (QGG). The details of the proposed method are outlined in Algorithm \ref{alg:APD}.}
\begin{algorithm}[H]
\caption{Generalized Accelerated Primal-Dual (GAPD) Method}\label{alg:APD}
\begin{algorithmic}[1]
    \STATE \textbf{Input:} $\{\tau_k, \sigma_k, \alpha_k, \beta_k, \theta_k\}_{k \geq 0}$, $(x_0, y_0) \in \mathcal{X} \times \mathcal{Y}$
    \STATE $(x_{-1}, y_{-1}) \gets (x_0, y_0)$
    \FOR{$k \geq 0$}
        \STATE \hspace{-1em}  \hspace{0.5em} $q_k^y \gets \nabla_y f(x_k, y_k) - \nabla_y f(x_{k-1}, y_{k-1})$;
        \STATE \hspace{-1em}  \hspace{0.5em} $y_{k+1} \gets \arg\min_{y \in {Y}} \left\{ -\langle \nabla_y f(x_k, y_k) + \alpha_k q_k^y, y \rangle + \frac{1}{\sigma_k} \rev{\bD_\cY(y, y_k)} \right\}$;
        \STATE \hspace{-1em} \hspace{0.5em} $q_k^x \gets \nabla_x f(x_k, y_k) - \nabla_x f(x_{k-1}, y_{k-1})$;
        \STATE \hspace{-1em}  \hspace{0.5em} $s_k \gets \theta_k \nabla_x f(x_k, y_{k+1}) + (1 - \theta_k) \nabla_x f(x_k, y_k) + \beta_k q_k^x$;
        \STATE \hspace{-1em}  \hspace{0.5em} $x_{k+1} \gets \arg\min_{x \in {X}} \left\{ \langle s_k, x \rangle + \frac{1}{\tau_k} \rev{\bD_\cX(x, x_k)} \right\}$;
    \ENDFOR
\end{algorithmic}
\end{algorithm}
\noindent\aj{The GAPD algorithm iteratively updates the primal and dual variables to solve saddle-point problems. It begins with initial  points $(x_0, y_0)$ and predefined parameter sequences $\{\tau_k, \sigma_k, \alpha_k, \beta_k, \theta_k\}_{k \geq 0}$. At each iteration, the algorithm computes the momentum term $q_k^y$, which represents the change in the gradient $\nabla_y f(x, y)$ over consecutive iterations, and updates the dual variable $y_{k+1}$ using a proximal step with respect to the Bregman distance $\bD_{\mathcal{Y}}(y, y_k)$. Similarly, the momentum term $q_k^x$ is calculated to capture the change in the gradient $\nabla_x f(x, y)$, and an aggregate gradient $s_k$ is constructed by combining gradients and momentum terms. The primal variable $x_{k+1}$ is then updated via a proximal step with respect to the Bregman distance $\bD_{\mathcal{X}}(x, x_k)$}. {Note that GAPD encompasses both OGDA in \cite{mokhtari2020b} and APD in \cite{hamedani2021primal}: when $\theta_k = 0$, GAPD reduces to OGDA, and when $\theta_k = 1$ and $\beta_k=0$, GAPD corresponds to APD.} 

\aj{The next section discusses the algorithm's convergence properties, while detailed proofs for step-size selection are provided in Appendix \ref{append}.}

\section{Convergence Analysis}

First, we provide and prove a useful lemma that is necessary in our one-step analysis.
\begin{lemma}\label{argminineq}
 Let $\mathcal{X}$ be a finite-dimensional real normed vector space with norm $\|\cdot\|_{\mathcal{X}}$, $g : \mathcal{X} \to \mathbb{R} \cup \{+\infty\}$ be a closed convex function with convexity modulus $\mu \geq 0$ with respect to $\|\cdot\|_{\mathcal{X}}$, and $\bD_{\mathcal{X}} : \mathcal{X} \times \mathcal{X} \to \mathbb{R}_+$ be a Bregman distance function corresponding to a strictly convex function $\varphi : \mathcal{X} \to \mathbb{R}$ that is differentiable on an open set containing $\text{dom } g$. Given $\bar{x} \in \text{dom } g$ and $t > 0$, let
\begin{equation}\label{tsengopt}
    x^+ = \arg\min_{x \in \mathcal{X}} g(x) + t \bD_{\mathcal{X}}(x, \bar{x}).
\end{equation}
Then for all $x \in \mathcal{X}$, the following inequality holds:
\begin{equation}
    g(x) + t\bD_{\mathcal{X}}(x, \bar{x}) \geq g(x^+) + t\bD_{\mathcal{X}}(x^+, \bar{x}) + t\bD_{\mathcal{X}}(x, x^+) + \frac{\mu}{2} \|x - x^+\|^2_{\mathcal{X}}.
\end{equation}
\end{lemma}
\textit{Proof.}
\rev{
This result is a trivial extension of Property 1 in \cite{Tseng08_1J}. The first-order optimality condition for {(\ref{tsengopt})} implies that
\[
0 \in \partial g(x^+) + t \nabla_x \bD_{\mathcal{X}}(x^+, \bar{x}),
\]
where $\nabla_x \bD_\mathcal{X}$ denotes the partial gradient with respect to the first argument. Note that for any $x \in \text{dom } g$, we have
$\nabla_x \bD_{\mathcal{X}}(x, \bar{x}) = \nabla \varphi(x) - \nabla \varphi(\bar{x})$. 
Hence,
$t (\nabla \varphi(\bar{x}) - \nabla \varphi(x^+)) \in \partial g(x^+)$. 
Using the convexity inequality for $g$, we get
\[
g(x) \geq g(x^+) + t \langle \nabla \varphi(\bar{x}) - \nabla \varphi(x^+), x - x^+ \rangle + \frac{\mu}{2} \|x - x^+\|_{\mathcal{X}}^2.
\]
Adding $t \bD_{\mathcal{X}}(x, \bar{x})$ to both sides,
\[
g(x) + t \bD_{\mathcal{X}}(x, \bar{x}) \geq g(x^+) + t \bD_\mathcal{X}(x, \bar{x}) + t \langle \nabla \varphi(\bar{x}) - \nabla \varphi(x^+), x - x^+ \rangle + \frac{\mu}{2} \|x - x^+\|_{\mathcal{X}}^2.
\]
Applying the three-point identity for Bregman distances,
\[
\bD_{\mathcal{X}}(x, \bar{x}) = \bD_{\mathcal{X}}(x, x^+) + \bD_{\mathcal{X}}(x^+, \bar{x}) + \left\langle \nabla \varphi(x^+) - \nabla \varphi(\bar{x}), x - x^+ \right\rangle,
\]
which implies
\[
\langle \nabla \varphi(x^+) - \nabla \varphi(\bar{x}), x - x^+ \rangle = \bD_{\mathcal{X}}(x, \bar{x}) - \bD_{\mathcal{X}}(x, x^+) - \bD_{\mathcal{X}}(x^+, \bar{x}).
\]
Substituting into the previous inequality gives
\begin{align*}
g(x) + t \bD_{\mathcal{X}}(x, \bar{x}) &\geq g(x^+) + t \bD_\mathcal{X}(x, \bar{x}) \revf{-} t \left( \bD_\mathcal{X}(x, \bar{x}) - \bD_\mathcal{X}(x, x^+) - \bD_\mathcal{X}(x^+, \bar{x}) \right) \\
&\quad + \frac{\mu}{2} \|x - x^+\|_{\mathcal{X}}^2 \\
&\revf{=} g(x^+) + t \bD_{\mathcal{X}}(x^+, \bar{x}) + t \bD_{\mathcal{X}}(x, x^+) + \frac{\mu}{2} \|x - x^+\|_{\mathcal{X}}^2,
\end{align*}
as claimed.
}
\qed

Next, we provide a one-step analysis of the proposed algorithm. 
\begin{lemma}\label{lem:one-step}
Suppose Assumption \ref{assumption:convex-concave-smoothness} holds.
Let $\{(x_k, y_k)\}_{k\geq 0}$ be the sequence generated by GAPD in Algorithm \ref{alg:APD}. Moreover, for any $k \geq 0$ define 
$z_k \triangleq [(x_k)^\top \ (y_k)^\top]^\top$, $F_k \triangleq F(z_k)= [\nabla_x f(x_k,y_k)^\top \ -\nabla_y f(x_k,y_k)^\top]^\top$, $q_k \triangleq [(q_k^x)^\top \ -(q_k^y)^\top]^\top$. Then for any $k\geq 0$ 
\begin{align*}
    &(1-\theta_k)\langle F(z_{k+1}), z_{k+1}-z\rangle
    + \theta_k\bigl(f(x_{k+1},y)-f(x,y_{k+1}) \bigr) \\
    &\leq \bD_\cZ^{A_k}(z,z_k)
    - \bD_\cZ^{A_k}(z,z_{k+1}) + \langle F(z_{k+1})-F(z_k), z_{k+1}-z\rangle_{C_k}
    - \langle q_k, z_k-z\rangle_{B_k} + \Lambda_k .
\end{align*}
where $\Lambda_k\triangleq \langle q_k,z_k-z_{k+1} \rangle_{B_k}-\bD_\cZ^{D_k}(z_{k+1},z_k)$ and the matrices are defined as $A_k\triangleq \diag(\frac{1}{\tau_k}\bI_n,\frac{1}{\sigma_k}\bI_m)$, $B_k\triangleq \diag(\beta_k\bI_n,\alpha_k\bI_m)$, $C_k\triangleq \diag((1-\theta_k)\bI_n,\bI_m)$, and $D_k\triangleq \diag((\frac{1}{\tau_k}-\theta_kL_{xx})\bI_n,\frac{1}{\sigma_k}\bI_m)$. 
\end{lemma}

\begin{proof}
Applying Lemma \ref{argminineq} for the update of $x_{k+1}$ in Algorithm \ref{alg:APD} we obtain the following inequality for any $x\in X$
\begin{align}
0 &\leq \fprod{s_k,x-x_{k+1}}  +\frac{1}{\tau_k}\Big[\bD_{\cX}(x,x_k)-{\bD_{\cX}(x,x_{k+1})}-\bD_{\cX}(x_{k+1},x_k)\Big]\nonumber\\
&=\theta_k\fprod{\grad_{x}f(x_k,y_{k+1}),x-x_{k+1}}+\fprod{(1-\theta_k)\grad_{x}f(x_k,y_k)+\beta_kq_k^x,x-x_{k+1}}\nonumber \\
&\quad +\frac{1}{\tau_k}\Big[\bD_{\cX}(x,x_k)-{\bD_{\cX}(x,x_{k+1})}-\bD_{\cX}(x_{k+1},x_k)\Big]\nonumber.
\end{align} 
\rev{For notational convenience, we define \( \nabla_x f_k \triangleq \nabla_x f(x_k, y_k) \) and \( \nabla_y f_k \triangleq \nabla_y f(x_k, y_k) \) for any $k\geq 0$, in the rest of the proof.}
Adding \((1 - \theta_k) \langle \nabla_x f_{k+1}, x_{k+1} - x \rangle\) to both sides of the inequality above, then rearranging the terms in the inner product imply that 
\begin{align}\label{eq:sc-x}
&(1-\theta_k)\fprod{\grad_x f_{k+1},x_{k+1}-x}\\ \nonumber&\leq (1-\theta_k)\fprod{\grad_x f_{k+1}-\grad_x f_k,x_{k+1}-x}  - \beta_k\fprod{q_k^x,x_k-x} + \beta_k\fprod{q_k^x,x_k-x_{k+1}}  \nonumber\\
&\quad + \theta_k\fprod{\grad_x f(x_k,y_{k+1}),x-x_{k+1}} + \frac{1}{\tau_k}\Big[\bD_{\cX}(x,x_k)-\bD_{\cX}(x,x_{k+1})-\bD_{\cX}(x_{k+1},x_k)\Big]\nonumber.
\end{align}
Next, by convexity of $f$ in $x$ and the Lipschitz continuity of $\nabla_x f$,
\begin{align}\label{eq:sc-x2}
    f(x,y_{k+1})&\geq f(x_k,y_{k+1})+\langle \nabla_x f(x_k,y_{k+1}),x-x_k \rangle \nonumber \\
    &\geq f_{k+1}+\langle \nabla_x f(x_k,y_{k+1}),x-x_{k+1}\rangle -L_{xx}\bD_\cX (x_{k+1},x_k)\nonumber\\
    &\geq f(x_{k+1},y)+\langle \nabla_y f_{k+1},y_{k+1}-y\ \rangle +\langle \nabla_x f(x_k,y_{k+1}),x-x_{k+1} \rangle \nonumber\\
    &\quad -L_{xx}\bD_\cX (x_{k+1},x_k),
\end{align}
where we used concavity of $f$ in $y$ and the Bregman generator $\varphi$ being 1-strongly convex to obtain the last inequality. Multiplying the above by $\theta_k$ and adding to \ref{eq:sc-x} gives us,
\begin{align}
    &(1-\theta_k)\fprod{\grad_x f_{k+1},x_{k+1}-x}+\theta_k (f(x_{k+1},y)-f(x,y_{k+1}))+\theta_k\langle \nabla_y f_{k+1},y_{k+1}-y \rangle\nonumber\\
    &\leq (1-\theta_k)\fprod{\grad_x f_{k+1}-\grad_x f_k,x_{k+1}-x}  - \beta_k\fprod{q_k^x,x_k-x} + \beta_k\fprod{q_k^x,x_k-x_{k+1}} \nonumber \\ 
    &\quad +\frac{1}{\tau_k}\bD_\cX (x,x_k)-\frac{1}{\tau_k}\bD_\cX (x,x_{k+1})+(\theta_k L_{xx}-\frac{1}{\tau_k})\bD_\cX (x_{k+1},x_k)
\end{align}

Similarly, applying Lemma \ref{argminineq} for the update of $y_{k+1}$ we can obtain the following inequality for any $y\in Y$,
\[
\langle \nabla_y f_k,y-y_{k+1}\rangle +\alpha_k\langle q_k^y,y-y_{k+1}\rangle \leq \frac{1}{\sigma_k}\Big[ \bD_{\mathcal{Y}}(y,y_k)-\bD_{\mathcal{Y}}(y_{k+1},y_k)-\bD_{\mathcal{Y}}(y,y_{k+1}) \Big],
\]
which by rearranging the terms and adding $\fprod{\nabla_y f_{k+1},y-y_{k+1}}$ to both sides we obtain 
\begin{align}\label{eq:sc-y}
&\fprod{\grad_y f_{k+1},y-y_{k+1}}\nonumber \\ &\leq \fprod{\grad_y f_{k+1}-\grad_y f_k,y-y_{k+1}}+\alpha_k\fprod{q_k^y,y_k-y}+\alpha_k\fprod{q_k^y,y_{k+1}-y_k} \nonumber\\
&\quad +\frac{1}{\sigma_k}\Big[\bD_{\cY}(y,y_k)-{\bD_{\cY}(y,y_{k+1})}-\bD_{\cY}(y_{k+1},y_k)\Big].
\end{align}
Adding \eqref{eq:sc-y} to \eqref{eq:sc-x2} gives the result. 

\end{proof}
\eyh{The results in Lemma \ref{lem:one-step} relied solely on convex-concave and smoothness assumptions. Now, we introduce additional structural assumptions to guarantee a linear convergence rate for the proposed Algorithm \ref{alg:APD}.} 
\begin{assumption}\label{assumption:QGG}
    The objective function $f$ satisfies either two-sided QFG or two-sided QGG with parameters $(\mu_x,\mu_y)>0$.
\end{assumption}
\begin{assumption}\label{assumption:Bregman}
    The Bregman distance generating functions $\psi_\cX$ and $\psi_\cY$ have Lipschitz continuous gradient with constants $L_{\psi_\cX}$ and $L_{\psi_\cY}$, respectively. 
\end{assumption}
\eyh{Next, we establish general conditions on the step-sizes and algorithm parameters necessary to prove convergence. These conditions ensure the formation of telescopic terms in the upper-bound of the gap function and enforce that the step-sizes are bounded by the relevant Lipschitz constants. Subsequently, we specify parameter choices that meet these conditions.}
\begin{assumption}\label{assumption:stepsize}
    The parameters of Algorithm \ref{alg:APD} $\{\tau_k,\sigma_k,\theta_k,\alpha_k,\beta_k\}_{k\geq 0}$ satisfy the following conditions:
\begin{subequations}\label{eq:stepsize-cond}
\begin{align}
    &t_k(1-\theta_k)=t_{k+1}\beta_{k+1},\quad t_k=t_{k+1}\alpha_{k+1}\label{eq:stepsize-cond-a}\\
    &\revv{t_{k+1}\Big[ \frac{1}{\tau_{k+1}} \Big]\leq t_k\Big[\varsigma_k \mu_x +\frac{1}{\tau_k}-\frac{2(1-\theta_k)}{\gamma}(L^2_{yx}+L^2_{xx}) \Big]},\nonumber\\
    &\revv{t_{k+1}\Big[ \frac{1}{\sigma_{k+1}} \Big]\leq t_k\Big[\varsigma_k \mu_y +\frac{1}{\sigma_k}-\frac{2(1-\theta_k)}{\gamma}(L^2_{xy}+L^2_{yy}) \Big]}\label{eq:stepsize-cond-b}\\
    &\revv{\Lambda_k +\frac{1}{2}\|q_{k+1}\|^2_{\Upsilon}
    +\frac{1}{2}\|z_{k}-z_{k+1}\|^2_{\gamma(1-\theta_k) \bL^2_{\psi}+\Upsilon^{-1} C_k^2 \bL^2_{\psi}} }\nonumber\\
    &\revv{\qquad -\frac{1}{2}\| q_{k}\|^2_{\Upsilon^{-1}B_k}
    +\frac{t_{k+1}}{2t_k}\|q_{k+1}\|^2_{\Upsilon^{-1}B_{k+1}}\leq 0}\label{eq:stepsize-cond-c}
\end{align}
\end{subequations}
where \revv{$\mathbf{S}_{f}\triangleq \diag\left\{2(L_{xx}^2+L_{yx}^2)\bI_n,2(L_{xy}^2+L_{yy}^2)\bI_m\right\}$, 
$\gamma>0$, $\bL_{\psi}=\diag\{L_{\psi_x}\bI_n,L_{\psi_y}\bI_m\}$}, \revf{$\Upsilon=\diag\{\upsilon_x\bI_n,\upsilon_y\bI_m\}$ with $\upsilon_x,\upsilon_y>0$,}
and $\varsigma_k>0$ such that $\varsigma_k=\theta_k$ when $f$ satisfies two-sided QFG and $\varsigma_k=2(1-\theta_k)$ when $f$ satisfies two-sided QGG.
\end{assumption}

\aj{Before presenting the main convergence result, we establish the non-expansivity of the Bregman projection under the Lipschitz continuity of the Bregman distance-generating function, a key component in our analysis.}
\begin{lemma}\label{lem:nonexpansive-proj}
Let $\varphi_\cU$ be a distance generating function as in Definition \ref{def-bregman} with the corresponding Bregman distance function $\bD_\cU(\cdot,\cdot)$. Let the Bregman projection be defined as in Definition \ref{def-bregman-projection}, and suppose Assumption \ref{assumption:Bregman} holds. Then for any $u,\hat u\in\cU$, 
    \begin{align*}
        \|\cP_{\cC}(u)-\cP_{\cC}(\hat u)\|_\cU\leq L_{\varphi_\cU}\norm{u-\hat u}_{\cU}.
    \end{align*}
\end{lemma}
\begin{proof}
The first-order optimality conditions for the Bregman projection for points $u, \hat u$ are given by:
\begin{align*}
&\langle \nabla\varphi (\cP_{\cC}(\hat u))-\nabla\varphi(\hat u),z-\cP_{\cC}(\hat u)\rangle \geq 0, \forall z\in \mathcal{C},\\ 
&\langle \nabla\varphi (\cP_{\cC}(u))-\nabla\varphi(u),w-\cP_{\cC}(u)\rangle \geq 0, \forall w\in \mathcal{C}.
\end{align*}
By definition, $u,\hat u \in \mathcal{U}$, and it follows that:
\begin{align*}
&\langle \nabla\varphi (\cP_{\cC}(\hat u))-\nabla\varphi(\hat u),\cP_{\cC}(u)-\cP_{\cC}(\hat u)\rangle \geq 0,\\ 
&\langle \nabla\varphi (\cP_{\cC}(u))-\nabla\varphi(u),\cP_{\cC}(u)-\cP_{\cC}(\revf{\hat u})\rangle \leq 0.
\end{align*}
Re-arranging and combining the inequalities results in,
\begin{align*}
\langle \nabla\varphi (\cP_{\cC}(u))-\nabla\varphi (\cP_{\cC}(\hat u)),\cP_{\cC}(u)-\cP_{\cC}(\hat u)\rangle \leq \langle \nabla\varphi(u)-\nabla\varphi(\hat u),\cP_{\cC}(u)-\cP_{\cC}(\hat u)\rangle.
\end{align*}
Utilizing the definition of 1-strong convexity for the left hand side, and Lipschitz continuity of the gradient of $\varphi$, the desired result can be concluded. 
\end{proof}

Now we are ready to prove the main convergence result of our paper. 
\begin{theorem}[Convergence rate of GAPD]\label{GAPDthm}
Let \( \{(x_k, y_k)\}_{k \geq 0} \) be the sequence generated by Algorithm \ref{alg:APD}. Suppose Assumptions \ref{assumption:convex-concave-smoothness}-\ref{assumption:stepsize} hold. Then, for any $K\geq 1$, the following holds
\begin{align}
    \bD_\cZ^{A_K-\Upsilon B_K} (\bar{z}_K,z_K)\leq \frac{t_0}{t_K}\bD^{A_0}_\cZ (\bar{z}_0,z_0)
\end{align}
where \revf{$t_0\triangleq 1$, $t_k \triangleq \prod_{j=1}^{k} \alpha_j^{-1}$ for $k\geq 1$}, $\bar{z}_k\triangleq \cP_{Z^*}(z_k)$ for any $k\geq 0$, and $A_k,B_k$ are defined in Lemma \ref{lem:one-step}. 
\end{theorem}
\begin{proof}
Recall the result of Lemma \ref{lem:one-step}, using our definition of $q_k$, and selecting $z=\bar{z}_k=(\bar{x}_k,\bar{y}_k)$
\begin{align}\label{eq:lem2-exp}
    &(1-\theta_k)\langle F(z_{k+1}), z_{k+1}-\bar{z}_k\rangle
    + \theta_k\bigl(f(x_{k+1},\bar{y}_k)-f(\bar{x}_k,y_{k+1})\bigr)\nonumber \\
    &\leq   \langle q_{k+1}, z_{k+1}-\bar{z}_k\rangle_{C_k}
    - \langle q_k, z_k-\bar{z}_k\rangle_{B_k} +\bD_\cZ^{A_k}(\bar{z}_k,z_k)
    - \bD_\cZ^{A_k}(\bar{z}_k,z_{k+1})  + \Lambda_k \nonumber\\
    &= \langle q_{k+1}, z_{k+1}-\bar{z}_{k+1}\rangle_{C_k}
    + \langle q_{k+1}, \bar{z}_{k+1}-\bar{z}_k\rangle_{C_k}
    - \langle q_k, z_k-\bar{z}_k\rangle_{B_k}\nonumber \\
    &\quad + \bD_\cZ^{A_k}(\bar{z}_k,z_k)
    - \bD_\cZ^{A_k}(\bar{z}_k,z_{k+1})
    + \Lambda_k. 
\end{align}
We first upper bound the inner product $\langle q_{k+1},\bar{z}_{k+1}-\bar{z}_k\rangle_{C_k}$. Using Cauchy-Schwartz, Lemma~\ref{lem:nonexpansive-proj}, and Young's inequality with $\Upsilon\succ 0$
\begin{align}\label{eq:qk+1bnd}
\langle q_{k+1},\bar{z}_{k+1}-\bar{z}_k\rangle_{C_k}&\leq \frac{1}{2}\|q_{k+1}\|^2_{\Upsilon}+\frac{1}{2}\|\bar z_{k+1}-\bar z_k\|^2_{\Upsilon^{-1}C_k^2 }\nonumber \\
&\leq \frac{1}{2}\|q_{k+1}\|^2_{\Upsilon}+\frac{1}{2}\|z_{k+1}-z_k\|^2_{\Upsilon^{-1}C_k^2 \bL_{\psi}^2},
\end{align}
where $\bL_{\psi}=\diag \{(L_{\psi_x}\bI_n,L_{\psi_y}\bI_m)\}$. 
Next, we lower bound the left-hand side of \eqref{eq:lem2-exp} by observing that for any $z_1^*,z_2^*\in Z^*$ the VI optimality conditions gives $\langle F(z_1^*), z_2^* - z_1^* \rangle \ge 0$ and $\langle F(z_2^*), z_1^* - z_2^* \rangle \ge 0$, hence, $\langle F(z_1^*)-F(z_2^*) , z_2^*-z_1^* \rangle\geq 0$ which together with the monotonicity of $F$ implies that $\langle F(z_1^*)-F(z_2^*) , z_2^*-z_1^* \rangle = 0$. Since the sum of two nonnegative terms is zero, it follows that $\langle F(z_1^*), z_2^*-z_1^* \rangle = 0$ and $\langle F(z_2^*), z_1^*-z_2^* \rangle = 0$. Therefore,
\begin{align}\label{eq:qfg-qgg-lb}
     &(1-\theta_k)\langle F(z_{k+1}), z_{k+1}-\bar{z}_k\rangle
    + \theta_k(f(x_{k+1},\bar{y}_k)-f(\bar{x}_k,y_{k+1}))\nonumber \\
    &=(1-\theta_k)\left(\langle F(z_{k+1}),z_{k+1}-\bar{z}_{k+1}\rangle +\langle F(z_{k+1}),\bar{z}_{k+1}-\bar{z}_k \rangle\right) \nonumber\\
    &\quad + \theta_k (f(x_{k+1},\bar{y}_k)-f(\bar{x}_k,y_{k+1}))\nonumber \\
    &\geq (1-\theta_k)\left(\langle F(z_{k+1})-F(\bar{z}_{k+1}),z_{k+1}-\bar{z}_{k+1}\rangle +\langle F(z_{k+1})-F(\bar z_{k+1}),\bar{z}_{k+1}-\bar{z}_k \rangle\right) \nonumber\\
    &\quad + \theta_k (f(x_{k+1},\bar{y}_k)-f(\bar{x}_k,y_{k+1}))\nonumber \\
    &\geq \bD^{\widetilde\bM_k}_\cZ(\bar{z}_{k+1},z_{k+1})+(1-\theta_k)\langle F(z_{k+1})-F(\bar{z}_{k+1}),\bar{z}_{k+1}-\bar{z}_{k} \rangle,
\end{align}
where $\widetilde\bM_k= \varsigma_k\bM$, the penultimate inequality follows from $\bar{z}_{k+1}\in Z^*$, hence, $\langle F(\bar{z}_{k+1}),z_{k+1}-\bar{z}_{k+1}\rangle\geq 0$ and $\langle F(\bar z_{k+1}),\bar{z}_{k+1}-\bar{z}_k \rangle = 0$. Moreover, in the last inequality we used Assumption \ref{assumption:QGG}\revf{; in particular, under two-sided QFG the term $(1-\theta_k)\langle F(z_{k+1})-F(\bar z_{k+1}),z_{k+1}-\bar z_{k+1}\rangle$ is nonnegative by the monotonicity of $F$, while under two-sided QGG the term $\theta_k(f(x_{k+1},\bar y_k)-f(\bar x_k,y_{k+1}))$ is nonnegative since $f(x_{k+1},\bar y_k)\geq f(\bar x_k,\bar y_k)\geq f(\bar x_k,y_{k+1})$}.
Note that we can move $(1-\theta_k)\langle F(z_{k+1})-F(\bar{z}_{k+1}),\bar{z}_{k+1}-\bar{z}_k \rangle$ to the right-hand side of \eqref{eq:lem2-exp}, and upper bound it as follows
\begin{align}\label{eq:Finnerbnd}
&(1-\theta_k)\langle F(z_{k+1})-F(\bar{z}_{k+1}),
\bar{z}_k-\bar{z}_{k+1} \rangle \nonumber\\
&\leq \frac{1-\theta_k}{2\gamma}\|F(z_{k+1})-F(\bar{z}_{k+1})\|^2  +\frac{\gamma(1-\theta_k)}{2}\|\bar{z}_k-\bar{z}_{k+1}\|^2\nonumber\\
&\leq \revf{\frac{1}{2}}\|z_{k+1}-\bar{z}_{k+1}\|^2_{\gamma^{-1}(1-\theta_k)\mathbf{S}_f} +\frac{1}{2}\|z_k-z_{k+1}\|^2_{\gamma(1-\theta_k)\bL_\psi^2} \nonumber\\
&\leq \bD_\cZ^{\revf{\gamma^{-1}(1-\theta_k)\mathbf{S}_f}}(\bar{z}_{k+1},z_{k+1})  +\frac{1}{2}\|z_k-z_{k+1}\|^2_{\gamma(1-\theta_k)\bL_\psi^2},
\end{align}
where we apply Young's inequality with $\gamma>0$ in the first inequality. Block Lipschitz continuity of the partial gradients of $f$, recalling that  $\mathbf{S}_{f}= \diag\left\{2(L_{xx}^2+L_{yx}^2)\bI_n,2(L_{xy}^2+L_{yy}^2)\bI_m\right\}$, along with the Bregman projection property (see Lemma~\ref{lem:nonexpansive-proj}) gives the second inequality. Finally, utilizing the fact that the squared Euclidean distance is a lower bound for the Bregman distance gives the final bound. Combining \eqref{eq:qk+1bnd}, \eqref{eq:qfg-qgg-lb}, and \eqref{eq:Finnerbnd} with \eqref{eq:lem2-exp} gives us,
\begin{align}
    &\bD^{\widetilde\bM_k}_\cZ(\bar{z}_{k+1},z_{k+1}) \nonumber\\
    &\leq \langle q_{k+1}, z_{k+1}-\bar{z}_{k+1}\rangle_{C_k}- \langle q_k, z_k-\bar{z}_k\rangle_{B_k}\nonumber \\
    &\quad +\bD^{A_k}_\cZ (\bar{z}_k,z_k)-\bD_\cZ^{A_k}(\bar{z}_k,z_{k+1})+\Lambda_k +\frac{1}{2}\|q_{k+1}\|^2_{\Upsilon}+\frac{1}{2}\|z_k-z_{k+1}\|^2_{\Upsilon^{-1}C_k^2 \bL_{\psi}^2}\nonumber \\
    &\quad+ \bD^{\revf{\gamma^{-1}(1-\theta_k)\mathbf{S}_f}}_\cZ (\bar{z}_{k+1},z_{k+1}) +\frac{1}{2}\|z_k-z_{k+1}\|^2_{\gamma(1-\theta_k)\bL_\psi^2}.\nonumber 
\end{align}
Using the fact that $\bar z_{k+1}$ has the smallest Bregman distance from $z_{k+1}$ \revv{we have $\bD^{A}_\cZ (\bar{z}_{k+1},z_{k+1})\leq \bD^{A}_\cZ (\bar{z}_k,z_{k+1})$, for any diagonal matrix $A\succ0$ (see the discussion after Definition \ref{optset})}, and by re-arranging the terms we obtain,
\begin{align}
    &\bD_\cZ^{\widetilde \bM_k +A_k-\revf{\gamma^{-1}}(1-\theta_k)\mathbf{S}_f} (\bar{z}_{k+1},z_{k+1}) \nonumber\\
    & \leq \bD_\cZ^{A_k}(\bar{z}_k,z_k) +\langle q_{k+1}, z_{k+1}-\bar{z}_{k+1}\rangle_{C_k}- \langle q_k, z_k-\bar{z}_k\rangle_{B_k} \nonumber \\
    &\quad +\Big( \Lambda_k +\frac{1}{2}\|q_{k+1}\|^2_{\Upsilon}+\frac{1}{2}\|z_{k}-z_{k+1}\|^2_{\gamma(1-\theta_k) \bL^2_{\psi}+\Upsilon^{-1} C_k^2 \bL^2_{\psi}} \Big).\nonumber
\end{align}
Multiplying both sides of the inequality by \revf{$t_k$ defined in the statement of the theorem}, using the step-size conditions from Assumption~\ref{assumption:stepsize}, and summing over $k=0$ to $K-1$,
\begin{align}
    \bD_\cZ^{t_KA_K}(\bar{z}_K,z_K)&\leq \bD_\cZ^{t_0 A_0} (\bar{z}_0,z_0) + \langle q_K,z_K-\bar{z}_K\rangle_{t_K B_K} -\langle q_0,z_0-\bar{z}_0 \rangle_{t_0 B_0} \nonumber\\
    &+\sum_{k=0}^{K-1}t_k\Big(\Lambda_k+\frac{1}{2}\|q_{k+1}\|^2_{\Upsilon}+\frac{1}{2}\|z_{k}-z_{k+1}\|^2_{\gamma(1-\theta_k) \bL^2_{\psi}+\Upsilon^{-1} C_k^2 \bL^2_{\psi}}   \Big)\nonumber\\ 
    &\leq \bD_\cZ^{t_0 A_0} (\bar{z}_0,z_0)+\frac{t_K}{2}\|q_K\|^2_{\Upsilon^{-1}B_K}+\frac{t_K}{2}\|z_K-\bar{z}_K\|^2_{\Upsilon B_K}\nonumber\\
    &+\sum_{k=0}^{K-1}t_k\Big(\Lambda_k+\frac{1}{2}\|q_{k+1}\|^2_{\Upsilon}+\frac{1}{2}\|z_{k}-z_{k+1}\|^2_{\gamma(1-\theta_k) \bL^2_{\psi}+\Upsilon^{-1} C_k^2 \bL^2_{\psi}}  \Big)\nonumber \\
    &\leq \bD_\cZ^{t_0 A_0} (\bar{z}_0,z_0)+\frac{t_K}{2}\|q_K\|^2_{\Upsilon^{-1}B_K} +t_K\bD_\cZ^{\Upsilon B_K}(\bar{z}_K,z_K)\nonumber\\
    &+\sum_{k=0}^{K-1}\Big(\frac{t_k}{2}\|q_k\|^2_{\Upsilon^{-1}B_k}-\frac{t_{k+1}}{2}\|q_{k+1}\|^2_{\Upsilon^{-1}B_{k+1}} \Big)\nonumber \\
    &=\bD_\cZ^{t_0 A_0} (\bar{z}_0,z_0)+t_K\bD_\cZ^{\Upsilon B_K}(\bar{z}_K,z_K).
\end{align}
Note that Young's inequality with $\Upsilon\succ 0$ was used in the second inequality along with the assumption that $q_0=0$. Further, 1-strong convexity of the Bregman generating function gives us the penultimate inequality. Next, the telescoping form in the penultimate inequality follows from \eqref{eq:stepsize-cond-c}, so that after summing over $k$ all intermediate terms cancel, the resulting terminal $q_K$ term cancels with the leading $q$ term, and the initial $q_0$ term vanishes by assumption. Finally, the result is obtain by moving the last Bregman term on the right-hand side to the left-hand side in the final equality. 
\end{proof}
\eyh{From the result of Theorem \ref{GAPDthm} we observe that $\dist(z_k, Z^*)$ converges to zero at a linear rate, {provided that $A_k - \Upsilon B_k$ }and $\alpha_k$ remain uniformly bounded from below. The following corollary demonstrates that this linear convergence rate can indeed be attained through an appropriate choice of parameters that satisfy Assumption \ref{assumption:stepsize}.}
\begin{corollary}\label{cor:stepresult}
\eyh{Let \( \{(x_k, y_k)\}_{k \geq 0} \) be the sequence generated by Algorithm \ref{alg:APD}. Suppose Assumptions \ref{assumption:convex-concave-smoothness}-\ref{assumption:Bregman} hold. The following selection of the algorithm's parameters satisfies Assumption \ref{assumption:stepsize}} \revf{with $\Upsilon=\bI$:}
\begin{align*}
    &\alpha = \max\{\alpha^*_x, \alpha^*_y\}, \quad \beta = \alpha(1 - \theta),\quad \revf{\theta\in(0,1]\ \text{if}\ \varsigma=\theta,\quad \theta\in[0,1)\ \text{if}\ \varsigma=2(1-\theta),}\\
    &t_k = \alpha^{-k}, \quad {\tau = \frac{2(1-\alpha)}{\alpha\varsigma\mu_x}, \quad \sigma =\frac{2(1-\alpha)}{\alpha\varsigma\mu_y},}
\end{align*}
where \revf{$\alpha^*_x,\alpha^*_y\in(0,1)$ are given in \eqref{eq:alpha-star-x-y} in Appendix \ref{append} and satisfy}
\begin{gather*}
\revf{\alpha^*_x=1-\Theta\left(\frac{\varsigma\mu_x}{\varsigma\mu_x+L_x^2+\theta L_{xx}+(1-\theta)\gamma L_\psi^2}\right),}\\
\revf{\alpha^*_y=1-\Theta\left(\frac{\varsigma\mu_y}{\varsigma\mu_y+L_y^2+\big(1+(1-\theta)\gamma\big)L_\psi^2}\right),}
\end{gather*}
and we define $L_x\triangleq \sqrt{L_{xx}^2+L_{yx}^2}$, $L_y\triangleq \sqrt{L_{yy}^2+L_{xy}^2}$, $L_{\max}\triangleq\max\{L_x,L_y\}$, $\mu_{\min}\triangleq\min\{\mu_x,\mu_y\}$\revf{, $\mu_{\max}\triangleq\max\{\mu_x,\mu_y\}$, $L_\psi\triangleq\max\{L_{\psi_\cX},L_{\psi_\cY}\}$, and \[\gamma\triangleq\max\Big\{\frac{4(1-\theta)}{\varsigma}\max\Big\{\frac{L_x^2}{\mu_x},\frac{L_y^2}{\mu_y}\Big\},1\Big\}.\]} 
Moreover, the iterates converge linearly to the optimal solution set, i.e., $\bD_\cZ(\bar{z}_K, z_K) \leq \bD^{\bR_K}_\cZ(\bar{z}_0, z_0)$, where \revf{\[\mathbf{R}_K=\frac{\alpha\varsigma}{2(1-\alpha)\min\{4L_x^2,L_\psi^2\}}\alpha^K\mathbf{M}.\]}
In particular, if $f$ satisfies two-sided QFG condition, then selecting $\theta=1$ implies that 
\begin{equation*}
\revf{\alpha=1-\Omega\left(\frac{\mu_{\min}}{\mu_{\max}+L_{\max}^2+L_{xx}+L_\psi^2}\right);}
\end{equation*}
moreover, if $f$ satisfies two-sided QGG condition, then for any $\theta\in[0,1)$ we have 
\begin{equation*}
\revf{\alpha=1-\Omega\left(\frac{(1-\theta)\mu_{\min}}{(1-\theta)\mu_{\max}+L_{\max}^2+\theta L_{xx}+\big(1+(1-\theta)L_{\max}^2\mu_{\min}^{-1}\big)L_\psi^2}\right).}
\end{equation*}
\end{corollary}
\begin{proof}
    The proof is provided in Appendix \ref{append}.
\end{proof}

\begin{remark}
\eyh{We established the linear convergence rate of GAPD under Assumption \ref{assumption:QGG}, demonstrating its versatility as a unified algorithm that encompasses APD \cite{hamedani2021primal} and OGDA \cite{mokhtari2020b} as special cases when $\theta=1$ and $\theta=0$, respectively. The linear convergence rate is influenced by the parameter $\varsigma > 0$, which equals $\theta$ if $f$ satisfies the two-sided QFG condition, and $2(1-\theta)$ if $f$ satisfies the two-sided QGG condition. This result underscores that APD achieves linear convergence under the two-sided QFG condition, whereas OGDA requires two-sided QGG condition.} 
\rev{
Other existing first-order methods, such as the extragradient (EG) method, can handle the general constrained saddle-point problem we study and more generally VIs, 
however, these approaches have been shown to achieve linear convergence only under the stronger assumption of strong convexity–strong concavity and more generally strong monotonicity -- see \cite{faechinei2003finite}. Furthermore, methods such as the alternating direction method of multipliers (ADMM) are primarily designed for convex constrained minimization problems. For example, \cite{hong2017linear} show that ADMM attains linear convergence when the variables are separable, the objective has a specific composite structure, and the constraint coefficient matrices are full rank. These conditions limit the applicability of ADMM, and such methods may not address the general convex–concave saddle-point problems considered in this paper.}  
\end{remark}

\section{Class of problems satisfying two-sided QFG/QGG}\label{classproblems}
In this section, we aim to investigate specific problems satisfying Assumption \ref{assumption:QGG}. To begin, we define a key constant that will be used throughout this section.
\begin{definition}[Hoffman Constant \cite{hoffman1952approximate}]
\label{def:hoffman}
Let \( \mathcal{P} = \{x \in \mathbb{R}^n : Ax = b, Cx \leq d\} \) be a non-empty polyhedron, where \( A \in \mathbb{R}^{p \times n} \), \( C \in \mathbb{R}^{m \times n} \), and \( b \in \mathbb{R}^p \), \( d \in \mathbb{R}^m \). For a given pair of norms \( \|\cdot\|_\alpha \) on \( \mathbb{R}^{p+m} \) and \( \|\cdot\|_\beta \) on \( \mathbb{R}^n \), and the corresponding dual norms \( \|\cdot\|_{\alpha^*} \) and \( \|\cdot\|_{\beta^*} \), the Hoffman constant \( \theta(A, C) > 0 \) is defined as the smallest constant satisfying the \textit{Hoffman inequality}:
\[
\|x - \bar{x}\|_\beta \leq \theta(A, C) \left\| \begin{bmatrix} Ax - b \\ [Cx - d]_+ \end{bmatrix} \right\|_\alpha \quad \forall x \in \mathbb{R}^n,
\]
where \( \bar{x} \) is the projection of \( x \) onto \( \mathcal{P} \) in the \( \|\cdot\|_\beta \)-norm, and \( [Cx - d]_+ \) denotes the component-wise maximum \( \max(0, Cx - d) \). \cite{klatte2020hoffman} showed that the Hoffman constant is equal to:
\begin{align*}
\theta(A,C) = 
\sup \left\{ 
\left\| 
\begin{bmatrix} 
u \\ 
v 
\end{bmatrix} 
\right\|_{\alpha^*} :
\begin{aligned}
&\left\| A^T u + C^T v \right\|_{\beta^*} = 1, \\
&\text{rows of } C \text{ correspond to nonzero components of } v, \\
&\text{rows of } A \text{ are linearly independent}
\end{aligned}
\right\}.
\end{align*}

If \( C = 0 \) (i.e., there are no inequality constraints), and \( A \in \mathbb{R}^{p \times n} \) has full row rank, the Hoffman constant simplifies to:
\[
\theta(A, 0) = \frac{1}{\sigma_{\min}(A)},
\]
where \( \sigma_{\min}(A) \) denotes the smallest nonzero singular value of \( A \).
\end{definition}
The Hoffman constant provides a measure of how constraint violations, measured in the $ \|\cdot\|_\alpha$-norm, bound the distance to the feasible set $\mathcal{P}$, measured in the $ \|\cdot\|_\beta$-norm (see \cite{hoffman1952approximate},\cite{necoara2019linear} for more details).

\medskip
\noindent \eyh{Let 
$h:\reals^p\to \reals,g:\reals^q\to \reals$ be two continuously differentiable and strongly convex {functions} with {moduli} $\kappa_x, \kappa_y > 0$. We consider the following class of constrained structured convex-concave saddle point problem:}
\begin{align}\label{eq:constrained-opt2}
    \min_{x\in X}\max_{y\in Y}~h(C_1 x) + \langle Ax,y \rangle -g(C_2y), 
\end{align}
where
$X = \{\ x\in\mathbb{R}^n \mid Gx\leq a \}, \: 
Y = \{\, y\in\mathbb{R}^m \mid Fy\leq c\}$, $G\in \mathbb R^{r\times n}$, $F\in\mathbb R^{s\times m}$, $A\in\reals^{m\times n}$, $C_1\in\reals^{p\times n}$, and $C_2\in\reals^{q\times m}$. Suppose a saddle point solution $(x^*,y^*)$ of \eqref{eq:constrained-opt2} exists. In the next proposition, we show that this class of functions belong to the two-sided QFG and QGG under certain assumptions. 

\begin{proposition}\label{prop:appl-2qfg-2qgg}
Consider problem \eqref{eq:constrained-opt2} and suppose that the matrices $G$ and $F$ have full-row rank. 
Assume that $h$ and $g$ are $\kappa_x$- and $\kappa_y$-strongly convex, respectively, and that there exist constants $\xi_1,\xi_2,\xi_3,\xi_4>0$ such that 
\begin{subequations}\label{eq:constraint-matrix-cond}
\begin{align}
&\xi_1 C_1^\top C_1\succeq A^\top A, && 
\xi_2 C_1^\top C_1\succeq \|\lambda^*\|^2 G^\top G,\\
&\xi_3 C_2^\top C_2\succeq A A^\top, && 
\xi_4 C_2^\top C_2\succeq \|\nu^*\|^2 F^\top F, 
\end{align}
\end{subequations}
where $\lambda^*\in\mathbb R^r,\nu^*\in\mathbb R^s$ denote the unique dual multipliers corresponding to the constraints in $X$ and $Y$, respectively \revf{(under the above conditions, these multipliers are unique and common to all saddle point solutions; see the proof)}. 
Further, let the Bregman generating functions be 
$\varphi_\cX(x)=\tfrac12\|x\|_2^{2}$ and $\varphi_\cY(y)=\tfrac12\|y\|_2^{2}$. 
Then the problem class \eqref{eq:constrained-opt2} satisfies two-sided QGG and two-sided QFG conditions with moduli
\[
\mu_x=\mu_y=
\frac{\min\{\kappa_x,\kappa_y\}}
{\theta(D,J)^2\max\{1+\xi_1+\xi_2,\;1+\xi_3+\xi_4\}}.
\]
\end{proposition}

\begin{remark}
    \eyh{The proof of this proposition is inspired by \cite{necoara2019linear}, but it differs in key aspects due to the distinct nature of saddle point problems compared to minimization problems, requiring additional attention to detail. For instance, while the optimal solution set of a convex minimization problem can be characterized by the set of points whose objective value equals the optimal value (a property utilized in \cite{necoara2019linear}), this property does not extend to saddle point problems as the optimal objective value can be attained at points that are not saddle point solutions, necessitating a more nuanced approach to the proof.} 
\end{remark}
\begin{proof}
Let $f(x,y)$ denote the objective function in \eqref{eq:constrained-opt2}. 
We first show that the solution set to problem \eqref{eq:constrained-opt2} is a polyhedral set. Note that due to convex-concave structure of \eqref{eq:constrained-opt2}, the saddle point solution set of \eqref{eq:constrained-opt2} denoted by $Z^*$ is a convex set; moreover, if $(x_1^*,y_1^*)\in Z^*$ and $(x_2^*,y_2^*)\in Z^*$ then $(x_i^*,y_j^*)\in Z^*$ for any $i,j\in\{1,2\}$. We claim that there exist unique vectors $(t^*_x,t^*_y)\in\reals^p\times\reals^q$ such that $C_1x^*=t^*_x$ and $C_2 y^*=t^*_y$ for any $(x^*,y^*)\in Z^*$. Let $f(x,y)=h(C_1x)+\fprod{Ax,y}-g(C_2y)$, and suppose $(x_1^*,y_1^*)\in Z^*$ and $(x_2^*,y_2^*)\in Z^*$ are two distinct saddle point solutions, then from the convexity of $f(\cdot,y^*_1)$ it follows that
\begin{equation}\label{eq:sp-convex}
    f\left(\frac{x_1^*+x_2^*}{2},y^*_1\right)\leq \frac{1}{2}f(x_1^*,y^*_1) + \frac{1}{2}f(x_2^*,y^*_1).
\end{equation}
Moreover, from the saddle point definition, we have $f(x_1^*,y_1^*)\leq f(x,y_1^*)$ and $f(x_2^*,y_1^*)\leq f(x,y_1^*)$ for any $x$. Therefore, selecting $x=\frac{x_1^*+x_2^*}{2}$ implies that $\frac{1}{2}f(x_1^*,y^*_1) + \frac{1}{2}f(x_2^*,y^*_1)\leq f\left(\frac{x_1^*+x_2^*}{2},y^*_1\right)$ which combining with \eqref{eq:sp-convex} implies that $f\left(\frac{x_1^*+x_2^*}{2},y^*_1\right)= \frac{1}{2}f(x_1^*,y^*_1) + \frac{1}{2}f(x_2^*,y^*_1)$. Therefore,
\begin{equation}\label{eq:sp-x1-x2}
    h\left(C_1(\frac{x_1^*+x_2^*}{2})\right)= \frac{1}{2}h(C_1x_1^*) + \frac{1}{2}h(C_1x_2^*).
\end{equation}
Using strong convexity of $h(\cdot)$ we have that 
\begin{equation}\label{eq:sp-sc}
    h\left(C_1(\frac{x_1^*+x_2^*}{2})\right)\leq \frac{1}{2}h(C_1x_1^*) + \frac{1}{2}h(C_1x_2^*)-\frac{\kappa_x}{8}\|C_1(x_1^*-x_2^*)\|^2.
\end{equation}
Strong convexity of $h$ implies that $C_1x^*_1=C_1x^*_2$. The same argument with concavity in $y$ and strong convexity of $g$ gives the mirror result $C_2y^*_1=C_2y^*_2$. Therefore there exist unique vectors $t_x^*\in\mathbb{R}^p$ and $t_y^*\in\mathbb{R}^q$ such that $C_1 x^* = t_x^*$ and $C_2 y^* = t_y^*$ for all $(x^*,y^*)\in Z^*$. Now, let us also define unique vectors $(T_x^*,T_y^*)$ such that $\grad h(C_1x^*)=T_x^*$ and $\grad g(C_2y^*)=T_y^*$. \revf{Moreover, $\xi_1 C_1^\top C_1\succeq A^\top A$ implies $\ker(C_1)\subseteq\ker(A)$, hence $A(x_1^*-x_2^*)=0$ for any $x_1^*,x_2^*\in X^*$; similarly, $\xi_3 C_2^\top C_2\succeq AA^\top$ implies $\ker(C_2)\subseteq\ker(A^\top)$, hence $A^\top(y_1^*-y_2^*)=0$ for any $y_1^*,y_2^*\in Y^*$. Therefore, $Ax^*$ and $A^\top y^*$ are also constant over $Z^*$.} 

For any $(x^*,y^*)\in Z^*$, differentiability of $h,g$ and convexity/concavity give the KKT equations \revf{(multipliers exist since the constraints are affine)}, where $\lambda^*,\nu^*$ are the associated KKT multipliers,
\begin{align*}
&C_1^\top \nabla h(C_1 x^*) + A^\top y^* \revf{+} G^\top \lambda^* = 0,\quad
-\,C_2^\top \nabla g(C_2 y^*) + A x^* - F^\top \nu^* = 0, \\ 
&Gx^* \le a,\;\; Fy^* \le c,\;\;
\lambda^* \ge 0,\;\; \nu^* \ge 0,\\
&\diag (\lambda^*) (a - Gx^*) = 0,\;\; \diag (\nu^*) (c - Fy^*) = 0.
\end{align*}
where $\diag (\lambda^*), \diag(\nu^*)$ denote diagonal matrices with $\lambda^*,\nu^*$ on diagonal entries. \revf{Since $G^\top\lambda^*=-C_1^\top T_x^*-A^\top y^*$ and $F^\top\nu^*=Ax^*-C_2^\top T_y^*$, where the right-hand sides are the same for all $(x^*,y^*)\in Z^*$, and $G$ and $F$ have full row rank, the multipliers $\lambda^*$ and $\nu^*$ are uniquely determined and common to all $(x^*,y^*)\in Z^*$.} Recalling that $z=[x^\top\ y^\top]^\top$, we can then define the block matrices for our $Z^*$,
\begin{align*}
D \triangleq \begin{bmatrix} C_1 & 0\\ A & 0 \\ \diag(\lambda^*)G & 0\\0 & C_2 \\ 0 & A^\top \\ 0 & \diag(\nu^*)F \end{bmatrix}, \quad b\triangleq \begin{bmatrix}
    t^*_x \\ C_2^\top T^*_y + F^\top \nu^* \\ \diag (\lambda^*) a\\ t_y^* \\\revf{-G^\top \lambda^*-C_1^\top T_x^*} \\ \diag (\nu^*) c
\end{bmatrix}, \quad J\triangleq \begin{bmatrix}  G & 0\\
0 & F\end{bmatrix}, \quad l =\begin{bmatrix}
    a \\ c
\end{bmatrix},
\end{align*}
to give us,
\begin{align*}
    Z^*=\{z\in Z \mid \text{there exists unique $\lambda^*,\nu^*$ such that}\quad  Dz= b, \quad Jz\leq l \}.
\end{align*}
Since $Z^*$ is defined by affine equality and inequality equations, we can conclude it is a polyhedron. Applying Def. $\ref{def:hoffman}$, there exists a constant $\theta(D,J)$ such that 
\begin{align*}
    \norm{z-\bar{z}}^2 \leq \theta^2 (D,J) ( \norm{Dz-b}^2 + \norm{[Jz-l]_+}^2).
\end{align*}
In fact, for any $z\in Z$ it follows that $[Jz-l]_+=0$. Therefore, expanding and re-arranging the right-hand side and noting that $b=D\bar z$, we obtain
\begin{align*}
\norm{z-\bar{z}}^2&\leq \theta^2 (D,J)\Big(
  \|C_1(x-\bar{x})\|^2 + \|C_2(y-\bar{y})\|^2 + \|A(x-\bar{x})\|^2 \\
&\quad
+ \|A^\top(y-\bar{y})\|^2
+ \|\lambda^*\|^2 \|G(x-\bar{x})\|^2
+ \|\nu^*\|^2 \|F(y-\bar{y})\|^2\Big).
\end{align*}
Next, utilizing our assumption that there exist constants $\xi_1,\xi_2,\xi_3,\xi_4 >0$ such that $\xi_1 C_1^\top C_1\succeq  A^\top A$, $\xi_2 C_1^\top C_1\succeq \norm{\lambda^*}^2 G^\top G$, $\xi_3 C_2^\top C_2\succeq A A^\top$, and $\xi_4 C_2^\top C_2\succeq \norm{\nu^*}^2 F^\top F$, we conclude that 
\begin{align*}
    \norm{z-\bar z}^2\leq \theta^2(D,J) (\max\{ 1+\xi_1+\xi_2,1+\xi_3+\xi_4 \})\Big( \norm{C_1 (x-\bar{x})}^2+ \norm{C_2 (y-\bar{y})}^2 \Big),
\end{align*}
which, after re-arranging the terms we obtain,
\begin{equation}\label{eq:C-block-lb}
\|C_1(x-\bar x)\|^2+\|C_2(y-\bar y)\|^2
\;\ge\;
\frac{\|z-\bar z\|^2}{\theta^2(D,J)\,\max\{1+\xi_1+\xi_2,\,1+\xi_3+\xi_4\}}.
\end{equation}

Now, from $\kappa_x,\kappa_y$-strong convexity of $h,g$ we have that
\begin{subequations}\label{eq:sc-C1-C2}
\begin{align}
    &\kappa_x\norm{C_1x-C_1\bar{x}}^2\leq \langle \nabla h(C_1x)-\nabla h(C_1\bar{x}),C_1x-C_1\bar{x}\rangle, \label{eq:sc-C1-x}\\
    &\kappa_y\norm{C_2y-C_2\bar{y}}^2\leq \langle \nabla g(C_2y)-\nabla g(C_2\bar{y}),C_2y-C_2\bar{y}\rangle.\label{eq:sc-C2-y}
\end{align}
\end{subequations}
Adding and subtracting $\fprod{A^\top (y-\bar{y}),x-\bar{x}}$ and $\fprod{A (x-\bar{x}),y-\bar{y}}$ to the right-hand side of \eqref{eq:sc-C1-x} and \eqref{eq:sc-C2-y} respectively, rearranging the terms, and combining the result with \eqref{eq:C-block-lb} imply the two-sided QGG as follows
\begin{align*}
\frac{2\min\{\kappa_x,\kappa_y\}}{\theta^2(D,J)\max\{(1+\xi_1+\xi_2),(1+\xi_3+\xi_4)\}}\bD_{\cZ}(z,\bar{z}) \leq   \langle F(z)-F(\bar{z}),z-\bar{z}\rangle.
\end{align*}
\eyh{Finally, note that we can replace the definitions of strong convexity in \eqref{eq:sc-C1-C2} with 
\begin{align*}
    \frac{\kappa_x}{2}\norm{C_1x-C_1\bar{x}}^2&\leq  h(C_1 x)-h(C_1\bar{x})-\langle\nabla h(C_1\bar{x}),C_1x-C_1\bar{x}\rangle \\
    &\revf{=h(C_1 x)-h(C_1\bar{x})+\langle A^\top \bar{y},x-\bar{x}\rangle+\langle \lambda^*,Gx-a\rangle}\\
    &\revf{\leq} h(C_1 x)-h(C_1\bar{x})+\langle A^\top \bar{y},x-\bar{x}\rangle,\\
    \frac{\kappa_y}{2}\norm{C_2y-C_2\bar{y}}^2&\leq g(C_2 y)-g(C_2\bar{y})-\langle\nabla g(C_2\bar{y}),C_2y-C_2\bar{y}\rangle\\
    &\revf{=g(C_2 y)-g(C_2\bar{y})-\langle A\bar{x},y-\bar{y}\rangle+\langle \nu^*,Fy-c\rangle}\\
    &\revf{\leq} g(C_2 y)-g(C_2\bar{y})-\langle A\bar{x},y-\bar{y}\rangle,
\end{align*}
\revf{where the equalities follow from the KKT conditions at $(\bar x,\bar y)$ together with $\diag(\lambda^*)(a-G\bar x)=0$ and $\diag(\nu^*)(c-F\bar y)=0$, and the inequalities follow from $\lambda^*,\nu^*\geq 0$, $Gx\leq a$, and $Fy\leq c$.}
Then, following the same steps by rearranging the terms results in 
\begin{align*}
\frac{\min\{\kappa_x,\kappa_y\}}{\theta^2(D,J)\max\{(1+\xi_1+\xi_2),(1+\xi_3+\xi_4)\}}\bD_{\cZ}(z,\bar{z}) \leq   f(x,\bar{y})-f(\bar{x},y).
\end{align*}}%
\revf{Finally, since $h(C_1x^*)$, $g(C_2y^*)$, $Ax^*$, and $A^\top y^*$ are constant over $Z^*$, we have $f(x,y^*)-f(x^*,y)=f(x,\bar y)-f(\bar x,y)$ for any $(x^*,y^*)\in Z^*$, so the two-sided QFG inequality \eqref{eq:qfg} holds for all $(x^*,y^*)\in Z^*$,} which completes the proof.
\end{proof}

\begin{remark}
The conditions \eqref{eq:constraint-matrix-cond} in Proposition~\ref{prop:appl-2qfg-2qgg} on the existence of 
$\xi_1,\xi_2,\xi_3,\xi_4>0$ are sufficient and easily checkable to ensure that the curvature induced by 
$C_1$ and $C_2$ dominates the effects of the coupling matrix $A$ and the active 
constraints associated with $(G,F)$. 
\revf{These relations are equivalent to the kernel inclusions $\ker(C_1)\subseteq\ker(A)$ and $\ker(C_2)\subseteq\ker(A^\top)$, together with $\ker(C_1)\subseteq\ker(G)$ when $\lambda^*\neq 0$ and $\ker(C_2)\subseteq\ker(F)$ when $\nu^*\neq 0$; that is, every direction along which the curvature induced by $C_1$ (resp. $C_2$) vanishes is also a direction along which the coupling and the active-constraint terms vanish,} thereby preventing directions in which the objective 
behaves linearly away from the solution set. In other words, $\xi_1 C_1^\top C_1\succeq A^\top A$ is equivalent to $\ker(\revf{C_1})\subseteq \ker(A)$, hence, in practice, one can select $\xi_1=\|A C_1^\dagger\|_2^2$ where $\revf{C_1^\dagger}$ denotes the Moore-Penrose pseudoinverse. The same reasoning applies analogously for the other constants.
\end{remark}

\begin{remark}\label{rem:qfg-qgg}
If problem \eqref{eq:constrained-opt2} is unconstrained, i.e., $X=\mathbb R^n$ and $Y=\mathbb R^m$, then the conditions and results of Proposition \ref{prop:appl-2qfg-2qgg} simplifies to the following:  
Suppose there exist $\nu_1,\nu_2 > 0$ such that 
$\nu_1 C_1^\top C_1 \succeq A^\top A$ and 
$\nu_2 C_2^\top C_2 \succeq AA^\top$, then the problem class satisfies two-sided QGG and two-sided QFG with moduli 
\[
\mu_x = \mu_y = 
\frac{\sigma^2_{\min}(D)\min\{\kappa_x,\kappa_y\}}
     {\max\{\,1+\nu_1,\;1+\nu_2\,\}}, 
\quad \text{where}\quad 
D \triangleq 
\begin{bmatrix}
C_1 & 0 \\
A & 0 \\
0 & C_2 \\
0 & A^\top
\end{bmatrix}.
\]
\end{remark}

\begin{remark}
For the unconstrained version of problem \eqref{eq:constrained-opt2}, the framework also captures a wide class of structured convex composite problems of the form
\begin{equation}\label{eq:opt-comp-convex}
    \min_{x\in\reals^n}~f_1(B_1x)+f_2(B_2x),
\end{equation}
where $f_1:\reals^p\to\reals$ is continuously differentiable and $\kappa_1$-strongly convex \revf{with an $L_{f_1}$-Lipschitz continuous gradient}, and $f_2:\reals^q\to\reals$ is \revf{$\ell_{f_2}$-strongly} convex with an $L_{f_2}$-Lipschitz continuous gradient\revf{, for some $0<\ell_{f_2}\leq L_{f_2}$}. 
By Fenchel duality, this problem can be equivalently written as
\begin{equation}\label{eq:fenchel-min-max}
    \min_{x\in\reals^n}\max_{y\in\reals^q}\; f_1(B_1x)+\langle B_2x,y\rangle - f_2^*(y),
\end{equation}
where $f_2^*$ denotes the Fenchel conjugate of $f_2$, defined as 
$f_2^*(y)=\max_{v\in\reals^q}\langle v,y\rangle - f_2(v)$. \revf{Since $f_2$ is $\ell_{f_2}$-strongly convex with an $L_{f_2}$-Lipschitz continuous gradient, $f_2^*$ is finite-valued, $(1/L_{f_2})$-strongly convex, and has a $(1/\ell_{f_2})$-Lipschitz continuous gradient; hence, \eqref{eq:fenchel-min-max} satisfies Assumption \ref{assumption:convex-concave-smoothness} with $L_{xx}=L_{f_1}\|B_1\|_2^2$, $L_{xy}=L_{yx}=\|B_2\|_2$, and $L_{yy}=1/\ell_{f_2}$.}
This representation is a special case of the unconstrained version of \eqref{eq:constrained-opt2} with
$h=f_1$, $g=f_2^*$, $C_1=B_1$, $C_2=I_q$, and $A=B_2$\revf{, with $\kappa_x=\kappa_1$ and $\kappa_y=1/L_{f_2}$}.
Hence, by Remark~\ref{rem:qfg-qgg}, the problem satisfies two-sided QGG and QFG, and Algorithm~\ref{alg:APD} achieves linear convergence.
This holds, for example, when $\nu_2=\|B_2\|_2^2$ and there exists $\nu_1>0$ such that 
$\nu_1 B_1^\top B_1\succeq B_2^\top B_2$.
\end{remark}

\section{\revv{Numerical Experiments}}\label{numerics}
In this section, we implement our proposed GAPD algorithm on two examples for various parameter selections $\theta$, and compare with the Gradient Descent Ascent (GDA) and Extragradient (EG) \cite{korpelevich1976extragradient} algorithms in several different dimensions\footnote{See \url{https://github.com/cmelcher/2qfg2qgg} for this section's code.}. Moreover, we empirically demonstrate the necessity of Assumption \ref{assumption:Bregman}.
\subsection{Bilinear Saddle-point Problem}
Following the previous section and Remark~\ref{rem:qfg-qgg}, we first consider the following unconstrained bilinear saddle point problem:
\begin{equation}\label{eq:quadratic-sp}
    \min_{x \in \mathbb R^n} \max_{y \in \mathbb R^m} 
    \; \tfrac{1}{2}\|C_1 x - b_1\|^2_2 + \langle A x, y\rangle - \tfrac{1}{2}\|C_2 y - b_2\|^2_2,
\end{equation}
where $C_1\in\mathbb R^{p\times n}$, $C_2\in\mathbb R^{q\times m}$, $A\in\mathbb R^{m\times n}$, $b_1 \in \mathbb{R}^{p}$, and $b_2 \in \mathbb{R}^{q}$. We consider $p<n$ and $q<m$ to ensure that $C_1$ and $C_2$ are not full column rank and therefore the objective is not strongly convex-strongly concave, while still satisfying two-sided QFG and two-sided QGG.

In order for this problem to fall into the class of problems considered in the previous section, in particular the existence of the constants $\nu_1$ and $\nu_2$, we generate the vectors $b_1,b_2$ and matrices $C_1$ and $C_2$ with independent standard Gaussian entries, and the matrices are normalized to have a unit spectral norm. 
We then set $A= C_2^\top R C_1$, 
where $R\in\mathbb{R}^{q\times p}$ is generated with standard Gaussian entries, then normalized to have a unit spectral norm. This ensures that $\nu_1,\nu_2$ exist, i.e., 
    $\nu_1=\|C_2^\top R\|_2^2,\quad  \nu_2=\|R C_1\|_2^2$. 

We choose the step size for GDA based on the result of \cite[Theorem 3.3]{zamani2024convergence}, which establishes a linear convergence rate for unconstrained saddle-point problems satisfying the two-sided QGG condition in Euclidean space. Specifically, we set the step size to $\frac{\mu}{\mu\max\{L_{xx},L_{yy}\}+L_{xy}^2}$. We use a similar step-size choice for EG. For our method, we choose the parameters according to Corollary~\ref{cor:stepresult}, where $\alpha_x$ and $\alpha_y$ are selected based on the prescribed big-$\Theta$ scaling and $\theta\in \{0,0.5,0.8,0.9,1\}$. 
We compare the performance of the methods across three problem dimensions, $(n,m,p,q)\in \{(900,720,540,432),$ $(1500,1200,900,720),$ $(3000,2400,1800,1440)\}$ in terms of the gradient residual normalized by the initial value as depicted in the plots from left to right in Figure \ref{fig:gapd_saddle}.

\begin{figure}[H]\label{fig1}
    \centering
 \includegraphics[width=\textwidth]{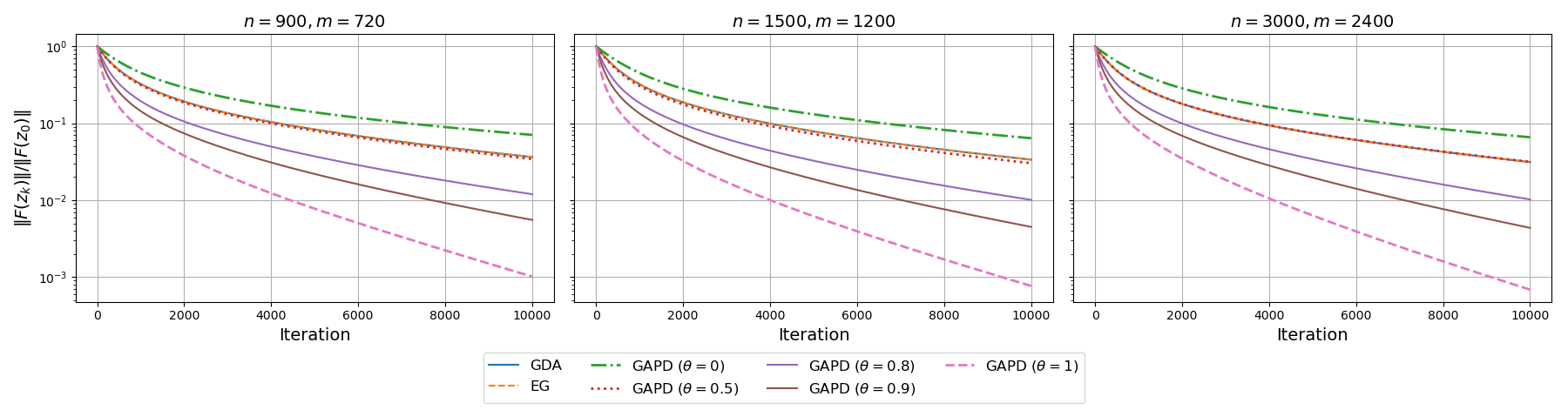}
 \vspace{-3mm}
    \caption{Convergence of GDA, EG, and GAPD variants on problem \eqref{eq:quadratic-sp}, 
    measured by the normalized gradient residual. From left to right, the problem dimensions are: $(n,m,p,q)= (900,720,540,432),(1500,1200,900,720),(3000,2400,1800,1440)$.}
    \label{fig:gapd_saddle}
\end{figure}

Across all three problem dimensions, all GAPD variants except $\theta=0$ outperform both EG and GDA. For this problem, larger values of $\theta$ consistently lead to better performance, with $\theta=1$ achieving the fastest convergence. This observation is consistent with our theoretical result in Corollary~\ref{cor:stepresult}, which indicates that choosing $\theta$ closer to $1$ allows for a larger step size. The similar behavior across the three dimensions further suggests that GAPD remains effective as the problem dimension increases.
\subsection{Nonbilinear Saddle-point Problem}
We next consider a nonbilinear convex-concave saddle-point problem inspired by robust power control in wireless networks. These models commonly combine transmit-power pricing with logarithmic quality-of-service or rate terms, and have been studied using game-theoretic and optimization formulations \cite{saraydar2001pricing,alpcan2004hybrid}. Closely related cellular-network energy models also couple transmit power and cell load through the Signal-to-Interference-plus-Noise Ratio (SINR) and the Shannon-type rate expression $\log(1+\mathrm{SINR})$ \cite{ho2015power}. Following this literature, we use a separable robust counterpart in which, for $R>0$, the variable $x\in[-R,R]^n$ represents signed transmission-amplitude adjustments, so that $x_i^2$ is proportional to the transmit-power or power-amplifier cost at link $i$. The variable $y\in[0,R]^n$ represents an adversarially selected normalized load or congestion level. The resulting problem can be formulated as 
\vspace{-0.2cm}
\begin{align}\label{eq:nonbilinear-sp}
     \min_{x \in [-R,R]^n} \max_{y \in [0,R]^n}
     \sum_{i=1}^n &a_i \rho(x_i)
    + \frac{1}{2}\sum_{i=1}^n \eta_i x_i^2 \log(1+y_i) -\sum_{i=1}^n b_i \rho(y_i),
\end{align}
where $\rho(\cdot)$ is a convex polynomial function, and $a_i,b_i,\eta_i>0$. In this experiment, we set $R=2$ and let $\rho(t)=\frac{t^2}{2}-\frac{t^4}{6}+\frac{t^6}{30}$, hence, $\rho''(t)=(t^2-1)^2\geq0$, so $\rho$ is convex. 
Note that the problem is not strongly convex-strongly concave on the chosen domain, because the $x$-block curvature can vanish at $(x_i,y_i)=(1,0)$ and the $y$-block curvature can vanish at $(x_i,y_i)=(0,1)$. Nevertheless, since $\rho(t)>0$ for $t\neq0$ and $a_i,b_i>0$, the saddle-point solution is unique and is given by $Z^*=\{(\mathbf{0},\mathbf{0})\}$ satisfying two-sided QGG/QFG. More specifically, since $\rho'(t)\,t = t^2\big(1-\tfrac{2}{3}t^2+\tfrac{1}{5}t^4\big)\geq \frac{4}{9}t^2$ for any $t\in [-R,R]$, we have  $\langle F(z)-F(z^*),\, z-z^*\rangle \geq \nu_F\Big(a_{\min}\|x\|_2^2 + b_{\min}\|y\|_2^2\Big)$ where $\nu_F=\frac{4}{9}$. Moreover, since $\rho(t)\geq \tfrac{7}{24}t^2$ for any $t\in [-R,R]$, we have $f(x,\mathbf{0}) - f(\mathbf{0},y)=\sum_{i=1}^n a_i\rho(x_i)+\sum_{i=1}^n b_i\rho(y_i) \geq \nu_f(a_{\min} \|x\|^2 + b_{\min} \|y\|^2)$ where $\nu_f=\frac{7}{24}$. 

In this experiment, we generate $a_i,b_i$ independently from a uniform distribution on $[0.1,1.1]$ and $\eta_i$ from a uniform distribution on $[1,10]$. Following Definitions \ref{qgg} and \ref{qfg}, the two-sided quadratic modulus is selected as $\mu_x=\nu\min(\{a_i\}_{i=1}^n)$ and $\revf{\mu_y}=\nu\min(\{b_i\}_{i=1}^n)$ where $\nu=\min\{\nu_F,2\nu_f\}$. 
Moreover, we choose the stepsizes for all the methods similar to the previous experiment. 

We compare the performance of GDA, EG, and the GAPD variants for $\theta\in\{0,0.5,0.8,0.9,1\}$ across three problem dimensions, \cm{$n\in\{1000,2000,5000\}$}, in terms of the distance of iterates from the saddle-point solution normalized by its initial value. The results are shown from left to right in Figure~\ref{fig:nonbilinear-conv} demonstrating that the GAPD variants converge linearly on this nonbilinear example and outperform the GDA and EG baselines. 
\begin{figure}[htbp]
    \centering
    \includegraphics[width=\textwidth]{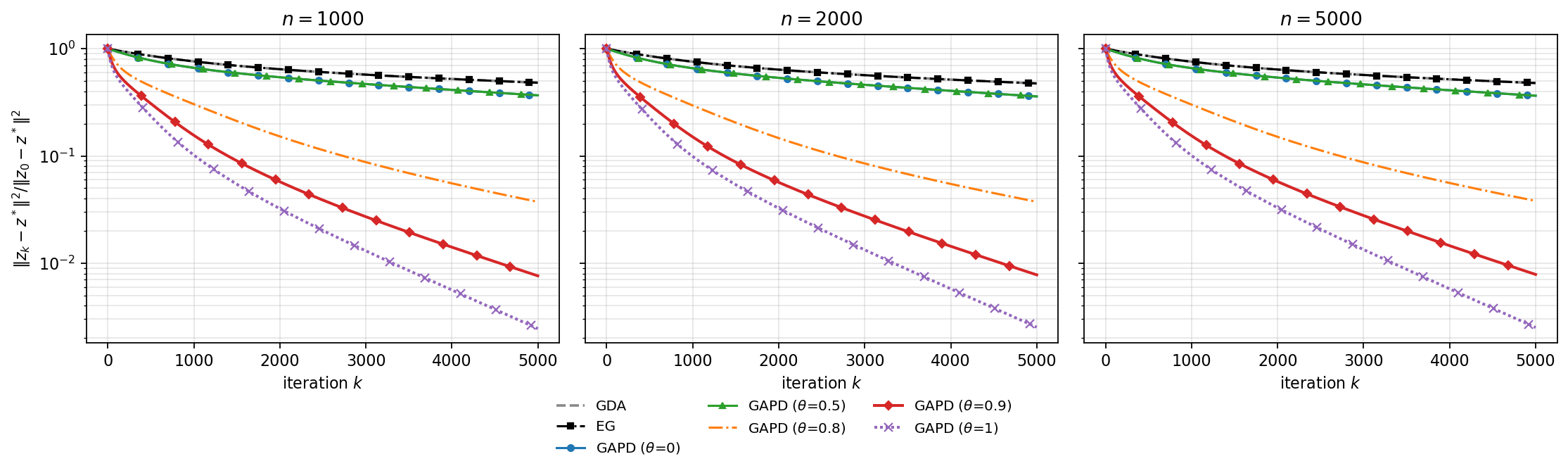}
    \vspace{-3mm}
    \caption{Convergence of GDA, EG, and GAPD variants on the nonbilinear robust wireless power-control problem \eqref{eq:nonbilinear-sp}, measured by the normalized distance from saddle-point solution. From left to right, the problem dimensions are \cm{$n=1000,2000,5000$.}}
    \label{fig:nonbilinear-conv}
\end{figure}
\subsection{Verification of Assumption \ref{assumption:Bregman}}
As a final numerical note, we would like to highlight the importance of Assumption \ref{assumption:Bregman}. Consider the constrained saddle point problem
\begin{align}
    \min_{x \in \Delta_n} \max_{y \in \Delta_m} 
    \; \tfrac{1}{2}\|C_1 x - b_1\|^2_2 + \langle A x, y\rangle - \tfrac{1}{2}\|C_2 y - b_2\|^2_2,
\end{align}
where $\Delta_n$ and $\Delta_m$ denote the probability simplex sets. We use the negative entropy Bregman distance function $\psi(x)=\sum_{i=1}^n x_i\log(x_i)$ to handle the simplex constraints. This Bregman distance choice is reasonable for a simplex-constrained problem, but the gradient of its generating function is not globally Lipschitz on the simplex set. Figure~\ref{fig:constrained-sp} compares GAPD under several choices of $\theta$ with linear and sublinear benchmark rates. We observe that the convergence behavior of GAPD is more closely aligned with the sublinear benchmark than with the linear benchmark. This numerical example illustrates that, without the assumption that the Bregman generating functions have Lipschitz continuous gradients, the linear convergence rate in Theorem~\ref{GAPDthm} may not generally be expected under our current framework.
\vspace{-3mm}
\begin{figure}[H]
    \centering
    \includegraphics[width=0.4\textwidth]{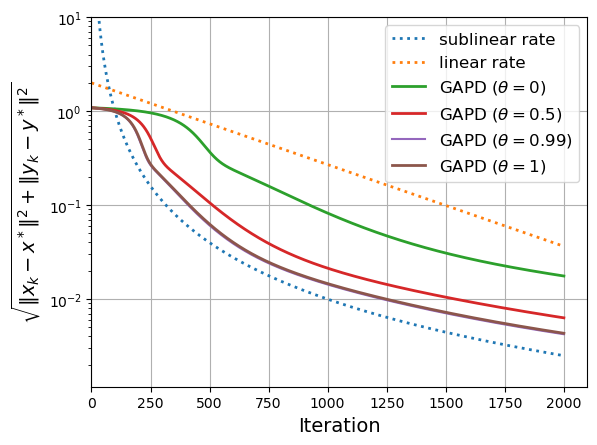}
    \caption{Comparison of the convergence of GAPD vs. linear and sublinear convergence behavior.}
    \label{fig:constrained-sp}
\end{figure}

\section{Conclusion and Future Directions}
In this paper, we explored saddle point (SP) problems with convex-concave structures under relaxed conditions, specifically two-sided QFG and two-sided QGG. These conditions extend the classic notion of strong convexity-strong concavity, enabling the analysis of a broader class of SP problems. We proposed the generalized accelerated primal-dual algorithm, which achieves a linear convergence rate under these relaxed conditions and unifies several existing methods. The theoretical results were complemented with examples of structured SP problems that satisfy two-sided QFG or two-sided QGG, illustrating the practical relevance of the proposed approach.

For future research, we aim to investigate additional examples and the implications of these growth conditions in the context of distributed optimization. Moreover, we plan to analyze the convergence behavior of SP problems under one-sided quadratic growth, exploring their applicability to constrained optimization problems.

\begin{appendices}

\section{Proof of Corollary \ref{cor:stepresult}}\label{append}
We show that the parameter selection in Corollary \ref{cor:stepresult} satisfies the conditions in \eqref{eq:stepsize-cond} with $\Upsilon=\bI$, i.e., $\upsilon_x=\upsilon_y=1$. Since the parameters are constant, $C_k=C\triangleq\diag((1-\theta)\bI_n,\bI_m)$, $B_k=B\triangleq\diag(\alpha(1-\theta)\bI_n,\alpha\bI_m)$, and $D_k=D\triangleq\diag((\frac{1}{\tau}-\theta L_{xx})\bI_n,\frac{1}{\sigma}\bI_m)$ for all $k\geq 0$.

\noindent\emph{Conditions \eqref{eq:stepsize-cond-a} and \eqref{eq:stepsize-cond-b}.} Since $t_k=\alpha^{-k}$ and $\beta=\alpha(1-\theta)$, we have $t_k(1-\theta)=t_{k+1}\beta$ and $t_k=t_{k+1}\alpha$, i.e., \eqref{eq:stepsize-cond-a} holds. Moreover, since $t_{k+1}/t_k=\alpha^{-1}$, \eqref{eq:stepsize-cond-b} is equivalent to
\begin{equation*}
\frac{1-\alpha}{\alpha\tau}\leq \varsigma\mu_x-\frac{2(1-\theta)}{\gamma}L_x^2,\qquad \frac{1-\alpha}{\alpha\sigma}\leq \varsigma\mu_y-\frac{2(1-\theta)}{\gamma}L_y^2.
\end{equation*}
If $\theta<1$, the choice of $\gamma$ (in particular, $\gamma\geq\frac{4(1-\theta)}{\varsigma}\max\{\frac{L_x^2}{\mu_x},\frac{L_y^2}{\mu_y}\}$) implies $\frac{2(1-\theta)}{\gamma}L_x^2\leq\frac{\varsigma\mu_x}{2}$ and $\frac{2(1-\theta)}{\gamma}L_y^2\leq\frac{\varsigma\mu_y}{2}$; if $\theta=1$, these terms vanish. Hence, \eqref{eq:stepsize-cond-b} holds whenever $\frac{1-\alpha}{\alpha\tau}\leq\frac{\varsigma\mu_x}{2}$ and $\frac{1-\alpha}{\alpha\sigma}\leq\frac{\varsigma\mu_y}{2}$, which hold with equality for $\tau=\frac{2(1-\alpha)}{\alpha\varsigma\mu_x}$ and $\sigma=\frac{2(1-\alpha)}{\alpha\varsigma\mu_y}$.

\noindent\emph{Condition \eqref{eq:stepsize-cond-c}.} Recall that $\Lambda_k=\langle q_k, z_k-z_{k+1}\rangle_{B}-\bD_\cZ^{D}(z_{k+1},z_k)$. We show below that $\frac{1}{\tau}\geq\theta L_{xx}$, so that $D\succeq 0$ and, by Remark \ref{bregineq}, $\bD_\cZ^{D}(z_{k+1},z_k)\geq\frac{1}{2}\|z_{k+1}-z_k\|^2_{D}$. Moreover, by Young's inequality, $\langle q_k, z_k-z_{k+1}\rangle_{B}\leq\frac{1}{2}\|q_k\|^2_{B}+\frac{1}{2}\|z_k-z_{k+1}\|^2_{B}$. Therefore, using $t_{k+1}/t_k=\alpha^{-1}$, the left-hand side of \eqref{eq:stepsize-cond-c} is bounded from above by
\begin{equation}\label{eq:cond-c-bound}
\frac{1}{2}\|q_{k+1}\|^2_{W}-\frac{1}{2}\|z_{k+1}-z_k\|^2_{D-B-\gamma(1-\theta)\bL_\psi^2-C^2\bL_\psi^2},
\end{equation}
where $W\triangleq \bI+\alpha^{-1}B=\diag\big((2-\theta)\bI_n,2\bI_m\big)$.
Since $W\preceq 2\bI$, using \eqref{eq:Lip-x}, \eqref{eq:Lip-y}, and $(a+b)^2\leq 2a^2+2b^2$, we obtain
\begin{align*}
\frac{1}{2}\|q_{k+1}\|^2_{W}&\leq \|\nabla_x f_{k+1}-\nabla_x f_k\|_{\cX^*}^2+\|\nabla_y f_{k+1}-\nabla_y f_k\|_{\cY^*}^2\\
&\leq \|z_{k+1}-z_k\|^2_{\mathbf{S}_f}=2L_x^2\|x_{k+1}-x_k\|^2_\cX+2L_y^2\|y_{k+1}-y_k\|^2_\cY.
\end{align*}
Combining this bound with \eqref{eq:cond-c-bound} and using $L_{\psi_\cX},L_{\psi_\cY}\leq L_\psi$, we conclude that \eqref{eq:stepsize-cond-c} holds if
\begin{equation}\label{eq:cx-cy}
\frac{1}{\tau}-\alpha(1-\theta)\geq c_x,\qquad \frac{1}{\sigma}-\alpha\geq c_y,
\end{equation}
where
\begin{equation*}
c_x\triangleq 4L_x^2+\theta L_{xx}+(1-\theta)(\gamma+1-\theta)L_\psi^2,\qquad c_y\triangleq 4L_y^2+\big(\gamma(1-\theta)+1\big)L_\psi^2.
\end{equation*} Note that \eqref{eq:cx-cy} implies $\frac{1}{\tau}-\theta L_{xx}\geq c_x-\theta L_{xx}\geq 4L_x^2>0$, which confirms $D\succeq 0$. Substituting $\frac{1}{\tau}=\frac{\alpha\varsigma\mu_x}{2(1-\alpha)}$ and $\frac{1}{\sigma}=\frac{\alpha\varsigma\mu_y}{2(1-\alpha)}$ and multiplying both inequalities in \eqref{eq:cx-cy} by $2(1-\alpha)>0$ shows that \eqref{eq:cx-cy} is equivalent to $Q_x(\alpha)\geq 0$ and $Q_y(\alpha)\geq 0$, where
\begin{align*}
Q_x(\alpha)&\triangleq \mathfrak{A}_x\alpha^2+\mathfrak{B}_x\alpha+\mathfrak{C}_x,\quad Q_y(\alpha)\triangleq \mathfrak{A}_y\alpha^2+\mathfrak{B}_y\alpha+\mathfrak{C}_y,\\
\mathfrak{A}_x&\triangleq 2(1-\theta),\quad \mathfrak{B}_x\triangleq\varsigma\mu_x+2c_x-2(1-\theta),\quad \mathfrak{C}_x\triangleq-2c_x,\\
\mathfrak{A}_y&\triangleq 2,\quad \mathfrak{B}_y\triangleq\varsigma\mu_y+2c_y-2,\quad \mathfrak{C}_y\triangleq-2c_y.
\end{align*}

Note that $c_x\geq 4L_x^2\geq 4L_{yx}^2>0$ by Assumption \ref{assumption:convex-concave-smoothness}, and $c_y\geq L_\psi^2\geq 1$, since a $1$-strongly convex function with an $L_\psi$-Lipschitz continuous gradient satisfies $L_\psi\geq 1$. Hence, $Q_x(0)=-2c_x<0$, $Q_y(0)=-2c_y<0$, $Q_x(1)=\varsigma\mu_x>0$, and $Q_y(1)=\varsigma\mu_y>0$. If $\theta<1$, $Q_x$ is a strictly convex quadratic function whose roots have the negative product $\mathfrak{C}_x/\mathfrak{A}_x$; hence, it has exactly one root in $(0,1)$, and $Q_x(\alpha)\geq 0$ for all $\alpha$ between this root and $1$. If $\theta=1$, $Q_x$ is affine with positive slope $\mathfrak{B}_x=\varsigma\mu_x+2c_x$, and the same conclusion holds. Similarly, $Q_y$ is a strictly convex quadratic function with exactly one root in $(0,1)$. These roots are
\begin{subequations}\label{eq:alpha-star-x-y}
\begin{align}
\alpha_x^*&=\frac{-\mathfrak{B}_x+\sqrt{\mathfrak{B}_x^2-4\mathfrak{A}_x\mathfrak{C}_x}}{2\mathfrak{A}_x}\ \text{ if }\theta<1,\qquad \alpha_x^*=\frac{2c_x}{\varsigma\mu_x+2c_x}\ \text{ if }\theta=1,\\
\alpha_y^*&=\frac{-\mathfrak{B}_y+\sqrt{\mathfrak{B}_y^2-4\mathfrak{A}_y\mathfrak{C}_y}}{2\mathfrak{A}_y}.
\end{align}
\end{subequations}
Consequently, $\alpha=\max\{\alpha_x^*,\alpha_y^*\}\in(0,1)$ satisfies \eqref{eq:cx-cy}, and thus \eqref{eq:stepsize-cond-c}.

\noindent\emph{Order of $1-\alpha$.} Since $Q_x(\alpha_x^*)=0$ and $Q_x(1)=\varsigma\mu_x$, we have the exact identity $\varsigma\mu_x=Q_x(1)-Q_x(\alpha_x^*)=(1-\alpha_x^*)\big(\mathfrak{A}_x(1+\alpha_x^*)+\mathfrak{B}_x\big)$. Since $\alpha_x^*\in(0,1)$, the last factor lies between $\mathfrak{A}_x+\mathfrak{B}_x=\varsigma\mu_x+2c_x$ and $2\mathfrak{A}_x+\mathfrak{B}_x=\varsigma\mu_x+2c_x+2(1-\theta)$. Moreover, $\gamma\geq 1$ and $L_\psi\geq 1$ imply $c_x\geq(1-\theta)\gamma L_\psi^2\geq 1-\theta$. Therefore,
\begin{equation*}
\frac{\varsigma\mu_x}{\varsigma\mu_x+4c_x}\leq 1-\alpha_x^*\leq\frac{\varsigma\mu_x}{\varsigma\mu_x+2c_x},
\end{equation*}
and similarly, since $\mathfrak{A}_y=2\leq 2c_y$,
\begin{equation*}
\frac{\varsigma\mu_y}{\varsigma\mu_y+4c_y}\leq 1-\alpha_y^*\leq\frac{\varsigma\mu_y}{\varsigma\mu_y+2c_y},
\end{equation*}
i.e., $1-\alpha_x^*=\Theta\big(\frac{\varsigma\mu_x}{\varsigma\mu_x+c_x}\big)$ and $1-\alpha_y^*=\Theta\big(\frac{\varsigma\mu_y}{\varsigma\mu_y+c_y}\big)$. Moreover, since $(1-\theta)^2\leq 1-\theta\leq(1-\theta)\gamma$, we have $L_x^2+\theta L_{xx}+(1-\theta)\gamma L_\psi^2\leq c_x\leq 4\big(L_x^2+\theta L_{xx}+(1-\theta)\gamma L_\psi^2\big)$ and $L_y^2+(1+(1-\theta)\gamma)L_\psi^2\leq c_y\leq 4\big(L_y^2+(1+(1-\theta)\gamma)L_\psi^2\big)$, which gives the expressions for $\alpha_x^*$ and $\alpha_y^*$ in Corollary \ref{cor:stepresult}. Since $1-\alpha=\min\{1-\alpha_x^*,1-\alpha_y^*\}$ and $\frac{a}{a+b}$ is increasing in $a>0$ and decreasing in $b>0$, the special cases follow: if $f$ satisfies two-sided QFG and $\theta=1$, then $\varsigma=1$ and the $\gamma$-terms vanish; if $f$ satisfies two-sided QGG and $\theta\in[0,1)$, then $\varsigma=2(1-\theta)$ and $(1-\theta)\gamma\leq(1-\theta)\big(1+2L_{\max}^2\mu_{\min}^{-1}\big)$.

\noindent\emph{Linear convergence.} By \eqref{eq:cx-cy}, $A_K-\Upsilon B_K=\diag\big((\frac{1}{\tau}-\alpha(1-\theta))\bI_n,(\frac{1}{\sigma}-\alpha)\bI_m\big)\succeq\diag(c_x\bI_n,c_y\bI_m)\succeq \min\{4L_x^2,L_\psi^2\}\bI$, and $A_0=\diag(\frac{1}{\tau}\bI_n,\frac{1}{\sigma}\bI_m)=\frac{\alpha\varsigma}{2(1-\alpha)}\mathbf{M}$. Therefore, since $t_0/t_K=\alpha^K$, Theorem \ref{GAPDthm} implies that $\min\{4L_x^2,L_\psi^2\}\bD_\cZ(\bar{z}_K,z_K)\leq \bD_\cZ^{A_K-\Upsilon B_K}(\bar{z}_K,z_K)\leq\frac{\alpha\varsigma\,\alpha^K}{2(1-\alpha)}\bD^{\mathbf{M}}_\cZ(\bar{z}_0,z_0)$, which gives the desired result. \qed

\end{appendices}

\bibliographystyle{plain}
\bibliography{sn-bibliography}

\begin{thebibliography}{10}

\bibitem{alpcan2004hybrid}
Tansu Alpcan and Tamer Ba\c{s}ar.
\newblock A hybrid systems model for power control in multicell wireless data
  networks.
\newblock {\em Performance Evaluation}, 57(4):477--495, 2004.

\bibitem{pmlr-v70-arjovsky17a}
Martin Arjovsky, Soumith Chintala, and L{\'e}on Bottou.
\newblock Wasserstein generative adversarial networks.
\newblock In {\em International conference on machine learning}, pages
  214--223. PMLR, 2017.

\bibitem{arrow1958studies}
Kenneth~J. Arrow, Leonid Hurwicz, and Hirofumi Uzawa.
\newblock {\em Studies in Linear and Non-linear Programming}.
\newblock Stanford University Press, Stanford, CA, 1958.

\bibitem{artacho2008characterization}
FJ~Arag{\'o}n Artacho and Michel~H Geoffroy.
\newblock Characterization of metric regularity of subdifferentials.
\newblock {\em Journal of Convex Analysis}, 15(2):365, 2008.

\bibitem{azizian2020tight}
Wa{\"\i}ss Azizian, Ioannis Mitliagkas, Simon Lacoste-Julien, and Gauthier
  Gidel.
\newblock A tight and unified analysis of gradient-based methods for a whole
  spectrum of differentiable games.
\newblock In {\em International conference on artificial intelligence and
  statistics}, pages 2863--2873. PMLR, 2020.

\bibitem{bacsar1998dynamic}
Tamer Ba{\c{s}}ar and Geert~Jan Olsder.
\newblock {\em Dynamic Noncooperative Game Theory}.
\newblock SIAM, Philadelphia, PA, 1998.

\bibitem{beck2017first}
Amir Beck.
\newblock {\em First-Order Methods in Optimization}, volume~25 of {\em MOS-SIAM
  Series on Optimization}.
\newblock SIAM, Philadelphia, PA, 2017.

\bibitem{bertsekas2009convex}
Dimitri~P. Bertsekas.
\newblock {\em Convex Optimization Theory}.
\newblock Athena Scientific, Nashua, NH, 2009.

\bibitem{chambolle2011first}
Antonin Chambolle and Thomas Pock.
\newblock A first-order primal-dual algorithm for convex problems with
  applications to imaging.
\newblock {\em Journal of Mathematical Imaging and Vision}, 40(1):120--145,
  2011.

\bibitem{chambolle2016ergodic}
Antonin Chambolle and Thomas Pock.
\newblock On the ergodic convergence rates of a first-order primal--dual
  algorithm.
\newblock {\em Mathematical Programming}, 159(1-2):253--287, 2016.

\bibitem{chen2014optimal}
Yunmei Chen, Guanghui Lan, and Yuyuan Ouyang.
\newblock Optimal primal-dual methods for a class of saddle point problems.
\newblock {\em SIAM Journal on Optimization}, 24(4):1779--1814, 2014.

\bibitem{condat2013primal}
Laurent Condat.
\newblock A primal--dual splitting method for convex optimization involving
  lipschitzian, proximable and linear composite terms.
\newblock {\em Journal of Optimization Theory and Applications},
  158(2):460--479, 2013.

\bibitem{daskalakis2017training}
Constantinos Daskalakis, Andrew Ilyas, Vasilis Syrgkanis, and Haoyang Zeng.
\newblock Training gans with optimism.
\newblock {\em arXiv preprint arXiv:1711.00141}, 2017.

\bibitem{drusvyatskiy2015quadratic}
Dmitriy Drusvyatskiy and Alexander~D Ioffe.
\newblock Quadratic growth and critical point stability of semi-algebraic
  functions.
\newblock {\em Mathematical Programming}, 153(2):635--653, 2015.

\bibitem{drusvyatskiy2013tilt}
Dmitriy Drusvyatskiy and Adrian~S Lewis.
\newblock Tilt stability, uniform quadratic growth, and strong metric
  regularity of the subdifferential.
\newblock {\em SIAM Journal on Optimization}, 23(1):256--267, 2013.

\bibitem{drusvyatskiy2018error}
Dmitriy Drusvyatskiy and Adrian~S Lewis.
\newblock Error bounds, quadratic growth, and linear convergence of proximal
  methods.
\newblock {\em Mathematics of operations research}, 43(3):919--948, 2018.

\bibitem{du2018linear}
Simon~S Du and Wei Hu.
\newblock Linear convergence of the primal-dual gradient method for
  convex-concave saddle point problems without strong convexity.
\newblock {\em arXiv preprint arXiv:1802.01504}, 2018.

\bibitem{faechinei2003finite}
F~Faechinei and JS~Pang.
\newblock Finite-dimensional variational inequalities and complementarity
  problems.
\newblock {\em New York: Springer-Verlag}, 2003.

\bibitem{goodfellow2014generativeadversarialnetworks}
Ian~J. Goodfellow, Jean Pouget-Abadie, Mehdi Mirza, Bing Xu, David
  Warde-Farley, Sherjil Ozair, Aaron Courville, and Yoshua Bengio.
\newblock Generative adversarial networks, 2014.

\bibitem{grimmer2019convergence}
Benjamin Grimmer.
\newblock Convergence rates for deterministic and stochastic subgradient
  methods without lipschitz continuity.
\newblock {\em SIAM Journal on Optimization}, 29(2):1350--1365, 2019.

\bibitem{hamedani2021primal}
Erfan~Yazdandoost Hamedani and Necdet~Serhat Aybat.
\newblock A primal-dual algorithm with line search for general convex-concave
  saddle point problems.
\newblock {\em SIAM Journal on Optimization}, 31(2):1299--1329, 2021.

\bibitem{ho2015power}
Chin~Keong Ho, Di~Yuan, Lei Lei, and Sumei Sun.
\newblock Power and load coupling in cellular networks for energy optimization.
\newblock {\em IEEE Transactions on Wireless Communications}, 14(1):509--519,
  2015.

\bibitem{hoffman1952approximate}
A.~J. Hoffman.
\newblock On approximate solutions of systems of linear inequalities.
\newblock {\em Journal of Research of the National Bureau of Standards},
  49(4):263--265, 1952.

\bibitem{hong2017linear}
Mingyi Hong and Zhi-Quan Luo.
\newblock On the linear convergence of the alternating direction method of
  multipliers.
\newblock {\em Mathematical Programming}, 162(1):165--199, 2017.

\bibitem{jiang2022generalized}
Ruichen Jiang and Aryan Mokhtari.
\newblock Generalized optimistic methods for convex-concave saddle point
  problems.
\newblock {\em arXiv preprint arXiv:2202.09674}, 2022.

\bibitem{Jin2025-go}
Qinian Jin.
\newblock {\em Lectures on nonsmooth optimization}.
\newblock Springer Nature, Cham, Switzerland, July 2025.

\bibitem{khaled2020better}
Ahmed Khaled and Peter Richt{\'a}rik.
\newblock Better theory for sgd in the nonconvex world.
\newblock {\em arXiv preprint arXiv:2002.03329}, 2020.

\bibitem{klatte2020hoffman}
Diethard Klatte.
\newblock Hoffman’s error bound for systems of convex inequalities.
\newblock In {\em Mathematical Programming with Data Perturbations}, pages
  185--199. CRC Press, Boca Raton, FL, 2020.

\bibitem{korpelevich1976extragradient}
Galina~M Korpelevich.
\newblock The extragradient method for finding saddle points and other
  problems.
\newblock {\em Matecon}, 12:747--756, 1976.

\bibitem{kotsalis2022simple}
Georgios Kotsalis, Guanghui Lan, and Tianjiao Li.
\newblock Simple and optimal methods for stochastic variational inequalities,
  i: operator extrapolation.
\newblock {\em SIAM Journal on Optimization}, 32(3):2041--2073, 2022.

\bibitem{meng2020aug}
Min Meng and Xiuxian Li.
\newblock Aug-pdg: Linear convergence of convex optimization with inequality
  constraints.
\newblock {\em arXiv preprint arXiv:2011.08569}, 2020.

\bibitem{mokhtari2020b}
Aryan Mokhtari, Asuman Ozdaglar, and Sarath Pattathil.
\newblock A unified analysis of extra-gradient and optimistic gradient methods
  for saddle point problems: Proximal point approach.
\newblock In {\em International Conference on Artificial Intelligence and
  Statistics}, pages 1497--1507. PMLR, 2020.

\bibitem{mokhtari2020a}
Aryan Mokhtari, Asuman~E Ozdaglar, and Sarath Pattathil.
\newblock Convergence rate of o(1/k) for optimistic gradient and extragradient
  methods in smooth convex-concave saddle point problems.
\newblock {\em SIAM Journal on Optimization}, 30(4):3230--3251, 2020.

\bibitem{monteiro2012iteration}
Renato~DC Monteiro and Benar~F Svaiter.
\newblock Iteration-complexity of a newton proximal extragradient method for
  monotone variational inequalities and inclusion problems.
\newblock {\em SIAM Journal on Optimization}, 22(3):914--935, 2012.

\bibitem{monteiro2011complexity}
Renato~DC Monteiro and Benar~Fux Svaiter.
\newblock Complexity of variants of tseng's modified fb splitting and
  korpelevich's methods for hemivariational inequalities with applications to
  saddle-point and convex optimization problems.
\newblock {\em SIAM Journal on Optimization}, 21(4):1688--1720, 2011.

\bibitem{necoara2019linear}
Ion Necoara, Yu~Nesterov, and Francois Glineur.
\newblock Linear convergence of first order methods for non-strongly convex
  optimization.
\newblock {\em Mathematical Programming}, 175:69--107, 2019.

\bibitem{nemirovski2004prox}
Arkadi Nemirovski.
\newblock Prox-method with rate of convergence $\mathcal{O}(1/t)$ for
  variational inequalities with lipschitz continuous monotone operators and
  smooth convex-concave saddle point problems.
\newblock {\em SIAM Journal on Optimization}, 15(1):229--251, 2004.

\bibitem{nesterov2007dual}
Yurii Nesterov.
\newblock Dual extrapolation and its applications to solving variational
  inequalities and related problems.
\newblock {\em Mathematical Programming}, 109(2):319--344, 2007.

\bibitem{nesterov2006solving}
Yurii Nesterov and Laura Scrimali.
\newblock Solving strongly monotone variational and quasi-variational
  inequalities.
\newblock Technical Report CORE Discussion Paper 2006/107, CORE, Université
  catholique de Louvain, 2006.
\newblock Available at SSRN: \url{https://ssrn.com/abstract=970903}.

\bibitem{nesterov2011solving}
Yurii Nesterov and Laura Scrimali.
\newblock Solving strongly monotone variational and quasi-variational
  inequalities.
\newblock {\em Discrete and Continuous Dynamical Systems}, 31(4):1383--1396,
  2011.

\bibitem{popov1980modification}
Leonid~Denisovich Popov.
\newblock A modification of the arrow-hurwitz method of search for saddle
  points.
\newblock {\em Mat. Zametki}, 28(5):777--784, 1980.

\bibitem{saraydar2001pricing}
Cem~U. Saraydar, Narayan~B. Mandayam, and David~J. Goodman.
\newblock Pricing and power control in a multicell wireless data network.
\newblock {\em IEEE Journal on Selected Areas in Communications},
  19(10):1883--1892, 2001.

\bibitem{sion1958general}
Maurice Sion.
\newblock On general minimax theorems.
\newblock {\em Pacific Journal of Mathematics}, 8(1):171--176, 1958.

\bibitem{Tseng08_1J}
P.~Tseng.
\newblock On accelerated proximal gradient methods for convex-concave
  optimization.
\newblock Available at
  \url{http://www.mit.edu/~dimitrib/PTseng/papers/apgm.pdf}, 2008.

\bibitem{tseng1995linear}
Paul Tseng.
\newblock On linear convergence of iterative methods for the variational
  inequality problem.
\newblock {\em Journal of Computational and Applied Mathematics},
  60(1-2):237--252, 1995.

\bibitem{tseng2000modified}
Paul Tseng.
\newblock A modified forward-backward splitting method for maximal monotone
  mappings.
\newblock {\em SIAM Journal on Control and Optimization}, 38(2):431--446, 2000.

\bibitem{yang2024variance}
Zhichun Yang, Fu-quan Xia, and Kai Tu.
\newblock Variance reduced moving balls approximation method for smooth
  constrained minimization problems.
\newblock {\em Optimization Letters}, 18(5):1253--1271, 2024.

\bibitem{yu2022fast}
Yaodong Yu, Tianyi Lin, Eric~V Mazumdar, and Michael Jordan.
\newblock Fast distributionally robust learning with variance-reduced min-max
  optimization.
\newblock In {\em International Conference on Artificial Intelligence and
  Statistics}, pages 1219--1250. PMLR, 2022.

\bibitem{zamani2024convergence}
Moslem Zamani, Hadi Abbaszadehpeivasti, and Etienne de~Klerk.
\newblock Convergence rate analysis of the gradient descent--ascent method for
  convex--concave saddle-point problems.
\newblock {\em Optimization Methods and Software}, 39(5):967--989, 2024.

\end{thebibliography}

\end{document}